\documentclass[11pt]{article}
\usepackage[utf8]{inputenc}

\usepackage[
  paper  = letterpaper,
  left   = 1.2in,
  right  = 1.2in,
  top    = 1in,
  bottom = 1in,
  ]{geometry}

\usepackage[
backend=bibtex,
style=numeric,
maxcitenames=1,
maxbibnames=16,
]{biblatex}
\DeclareNameAlias{author}{last-first}

\usepackage[bottom]{footmisc}
\usepackage{amsmath,amssymb,amsthm,mathtools,braket,dsfont}
\usepackage{physics}
\usepackage{float,graphicx,caption,subcaption,tikz}
\usepackage{enumitem}
\usepackage{adjustbox}
\usepackage{authblk}
\usepackage[hidelinks]{hyperref}

\newcommand{\R}{\mathbb{R}}

\numberwithin{equation}{section}
\newtheorem{theorem}{Theorem}[section]
\newtheorem{lemma}[theorem]{Lemma}
\newtheorem{proposition}[theorem]{Proposition}
\newtheorem{remark}[theorem]{Remark}

\theoremstyle{definition}
\newtheorem{definition}[theorem]{Definition}

\newcommand{\acks}[1]{\section*{Acknowledgements}#1}

\title{On the Synchronization Analysis of a Strong Competition Kuramoto Model}

\author[1]{Chun-Hsiung Hsia}
\author[2]{Chung-En Tsai}

\affil[1]{Institute of Applied Mathematical Sciences and National Center for Theoretical Sciences, National Taiwan University, Taipei 10617, Taiwan }
\affil[2]{Department of Computer Science and Information Engineering, National Taiwan University, Taipei 10617, Taiwan }

\date{\vspace{-5ex}}

\begin{document}
\maketitle

\begin{abstract}

When modeling  the classical Kuramoto model, one of the key features is the tendency to synchronize. Accordingly, the most well-adopted choice of the coupling function is the sine function. 
 Due to the oddness of the sine function, the synchronized frequency would be the average of all the natural frequencies. In this article, we study   the synchronization behaviors of the Kuramoto model with a pure competition coupling function. Namely, instead of the sine function, we choose  $\max \{0, \sin  \theta \}$ to be the  coupling function. This indicates the relation of  pure competition between oscillators. We prove asymptotical phase synchronization for identical oscillators and asymptotical frequency synchronization for non-identical oscillators under reasonable sufficient conditions. In particular, under our sufficient conditions, the synchronized frequency is the maximal frequency of all the natural frequencies. On the other hand, in the parameter regime which is out of the scope of the analysis of our theorems, it is possible that the synchronized frequency could be larger than the maximal frequency of the natural frequencies of all the oscillators. In this article, we also provide numerical experiments to support the analysis of our theorem and to demonstrate the aforementioned phenomenon. 
\end{abstract}


\section{Introduction}

 Synchronization phenomenon appears in a variety of natural systems, including pendulum clocks, triode generators, Josephson junction arrays, circadian rhythms, menstrual cycles, and fireflies \parencite{Acebron2005a,  Pikovsky2001a, Rodrigues2016a, Strogatz2000a}.

Among the mathematical models which describe the synchronous behaviors of a collection of oscillators, the one  proposed by \textcite{Kuramoto1975a,Kuramoto1984a} has received most attention. 
The model is formulated as a system of $N$ ordinary differential equations:
\begin{equation}\label{eq_kuramoto}
	\dot{\theta}_i(t) = \omega_i + \sum_{j=1}^N \Gamma(\theta_j(t) - \theta_i(t)), \quad  i  = 1, 2, \ldots, N.
\end{equation}
In this model, $\{\theta_i(t)\}_{i=1}^N$ is the set of oscillators, $\omega_i\in\R$ is the natural frequency of the $i$-th oscillator, and $\Gamma$ is a continuous and $2\pi$-periodic coupling function.
The oscillators are said to be \emph{identical} if all natural frequencies are the same, i.e.,
$$ \omega_i = \omega_j, \text{ for } i, j =1,2,3, \cdots, N. $$
\textcite{Kuramoto1975a} studied the sinusoidal coupling $\Gamma(\phi) = k\sin\phi$ for some coupling strength $k>0$, which is arguably the simplest and the most tractable case.
In the following, we refer 

\begin{equation}\label{eq_kuramoto_classical}
	\dot{\theta}_i(t) = \omega_i + k \sum_{j=1}^N \sin(\theta_j(t) - \theta_i(t)), \quad  i =1,2,\ldots, N.
\end{equation}
as the \emph{classical} Kuramoto model.

In this article, we analyze both phase synchronization and frequency synchronization. Here are some notation conventions and definitions that we shall use.
\paragraph{Notations.}
We denote
\begin{equation*}
  \Theta(t) \coloneqq (\theta_1(t),\ldots,\theta_N(t)) \,\, \text{ and } \,\,
 \Omega \coloneqq (\omega_1,\ldots,\omega_N).
 \end{equation*}
 For $X=(x_1, x_2, x_3, \cdots, x_N) \in \mathbb R^N$, we define the diameter function as
 $$D(X) \coloneqq \max_{1\leq i,j\leq N}\abs{x_i-x_j}.$$

\begin{definition}[Complete phase synchronization]
A solution $\Theta(t)$ to the system \eqref{eq_kuramoto} is said to achieve a complete phase synchronization asymptotically if for any $i,j \in \{1, 2, 3, \cdots, N\}$, there exists $n_{ij}\in \mathbb Z$ such that $\lim_{t\to\infty} (\theta_i(t) - \theta_j(t) - 2n_{ij}\pi) = 0$.
\end{definition}

\begin{definition}[Complete frequency synchronization]
A solution $\Theta(t)$ to the system \eqref{eq_kuramoto} is said to achieve a complete frequency synchronization asymptotically if for any $i,j\in \{1, 2, 3, \cdots, N\}$, we have $\lim_{t\to\infty} ( \dot{\theta}_i(t) - \dot{\theta}_j(t)) = 0$.
\end{definition}

For the classical Kuramoto model \eqref{eq_kuramoto_classical}, the critical coupling strength \parencite{Kuramoto1975a,Verwoerd2008a,Dorfler2011a}, bifurcation \parencite{Daido2016a}, and initial configurations that lead to synchronization \parencite{Benedetto2015a,  Chopra2009a, Dong2013a} have been studied.
It was proved in \parencite{Chopra2009a} that if the oscillators achieve complete frequency synchronization, the synchronized frequency equals the average of all the natural frequencies, i.e., 
\begin{equation}
\label{kuramoto-sync}
\lim_{t \to \infty} \dot{\theta_i}(t) = \frac {\sum_{j=1}^N \omega_j }{N} , \,\, \text { for }  i =1,2,3, \cdots, N. 
\end{equation}
Mathematically speaking, this is due to the fact that the sine function is an odd function; that is, $\Gamma(\phi) = -\Gamma(-\phi)$. To see this, summing up \eqref{eq_kuramoto_classical} over $i =1, 2, 3 \cdots, N$, we obtain 
\begin{equation}
\label{sumKuramoto}
\lim_{t \to \infty} \sum_{i=1}^N \dot{\theta_i}(t) = \sum_{j=1}^N \omega_j.
\end{equation}
If there exists $\omega \in \mathbb R$ such that $\lim_{t \to \infty} \dot{\theta}_i(t) = \omega$  for $i=1, 2, 3, \cdots, N$,  inferring from \eqref{sumKuramoto}, we obtain \eqref{kuramoto-sync}.
From the viewpoint of modeling, the choice of the sine function means that the leading one would have a tendency to slow down for the trailing ones and the trailing one would have a tendency to speed up, which indicates that the oscillators have a tendency to synchronize. This choice  of the coupling functions allows a structure of Lyapunov function for \eqref{eq_kuramoto_classical} which provides a systematic method to analyze the synchronization problem for \eqref{eq_kuramoto_classical}.
It turns out that the oddness of the coupling function is crucial for analyses based on Lyapunov function \parencite{Dong2013a, Hsia2019a, Van-Hemmen1993a} and the order parameter \parencite{Benedetto2015a}. Concerning the effect of time delay, we refer to \cite{Hsia2020a}. 
For other results, we refer the interested readers to the surveys of \textcite{Strogatz2000a}, \textcite{Acebron2005a} and \textcite{Rodrigues2016a}.

It is of mathematical interest to study the class of coupling functions that lead to synchronization.
For example, \textcite{Sakaguchi1986a,Ha2014a} studied the Kuramoto model with the phase-lag effect, which involves a non-odd coupling function $\Gamma(\phi) = k\sin(\phi + \alpha)$ for some $0<\abs{\alpha}<\pi/2$.
Coupling functions with bi-harmonics \parencite{Hansel1993a,Skardal2011a,Komarov2013a,Li2014a,Eydam2017a,Wang2017a} and higher-order harmonics \parencite{Daido1996a,Delabays2019a,Gong2019a,Li2019a} have also been studied in the literature.

In this article, we investigate the synchronization of a strong competition Kuramoto model beyond the sinusoidal coupling.
That is, we take $\Gamma(\theta) = k\max(0,\sin\theta)$ for some coupling strength $k>0$ as the coupling function.
This coupling function is  piecewise differentiable and non-odd.
With this choice of the coupling function, the model \eqref{eq_kuramoto} is rewritten as
\begin{equation}\label{eq_max_kuramoto}
	\dot{\theta}_i(t) = \omega_i + k\sum_{j=1}^N \max\{0, \sin(\theta_j(t) - \theta_i(t))\}, \quad   i = 1, 2, \cdots, N.
\end{equation}
We refer this model as the \emph{strong competition Kuramoto} model (SC Kuramoto model) afterwards.

\begin{remark}
It is not hard to check, for both \eqref{eq_kuramoto_classical} and \eqref{eq_max_kuramoto},  that if the oscillators achieve the complete phase synchronization, then they must be identical oscillators, i.e., $\omega_i = \omega_j,$ for $i, j  = 1, 2, 3, \cdots, N.$
\end{remark}

In the SC Kuramoto model, the $i$-th oscillator is affected  by the $j$-th oscillator only when the phase of the $j$-th oscillator is ``in front of'' the phase of the $i$-th oscillator.
This type of dynamic coupling has been considered by \textcite{Yang2020a,Yang2020b} and \textcite{Ho2024a} recently.
\textcite{Yang2020a,Yang2020b} are motivated by the phenomenon of the off-the-average synchronized frequency in several natural systems, such as the finger-tapping experiment and the applause of the audiences. However, they  conducted numerical experiments without rigorous mathematical analysis. \textcite{Ho2024a} develop a novel experimental assay that enables direct quantification of synchronization dynamics within mixtures of oscillating cell ensembles, for which the initial input frequency and phase distribution are known. Their results reveal a ``winner-takes-it-all'' synchronization outcome, i.e., the emerging collective rhythm matches one of the input rhythms. As shown in our main theorems, we use rigorous mathematical analysis to show that  the synchronized frequency is the largest natural frequency. 
\begin{theorem}\label{thm_phase_sync} Assume   $D(\Omega)=0$.
Let $\Theta(t)$ be a solution to \eqref{eq_max_kuramoto} with
 $D(\Theta(0)) < \pi$, then $\lim_{t\to\infty}D(\Theta(t)) = 0$;
that is, the oscillators achieve complete phase synchronization asymptotically.
\end{theorem}

\begin{theorem}\label{thm_freq_sync}
Assume  $k > D(\Omega)/\sin\delta$ for some $\delta \in (0,\pi/2)$\ and 
\begin{equation}
\label{omega_order}
\omega_1 \geq \omega_2 \geq \cdots \geq \omega_N.
\end{equation}
	Let $\Theta(t)$ be a solution to \eqref{eq_max_kuramoto} with $D(\Theta(0)) < \pi - \delta$, then 
	\begin{equation}
	\label{omega_max}
	\lim_{t \to \infty}\dot{\theta}_i(t) = \omega_1 = \max \{\omega_1, \omega_2, \cdots, \omega_N\} \,\,  \text{ for } i = 1, 2, 3, \cdots, N.
	\end{equation}
	In other words, the oscillators achieve a complete frequency synchronization asymptotically, and the synchronized frequency is the largest natural frequency.
\end{theorem}

\begin{remark}$~$
\begin{enumerate}
\item[(a)]
Under the assumption of Theorem~\ref{thm_freq_sync}, inferring from the well-ordering lemma (see Lemma~\ref{lem_order}), we see that no oscillator would be ``in front of'' $\theta_1(t)$ for all $ t \geq (\pi-2\delta)/(k\sin\delta - D(\Omega))$. This implies the synchronized frequency of $\theta_1$ would be $\omega_1$, which implies \eqref{omega_max}.\\
\item[(b)]
For the classical Kuramoto model \eqref{eq_kuramoto_classical}, it is known that $k > D(\Omega)/(N\sin\delta)$ suffices to achieve frequency synchronization \parencite{Chopra2009a}.
So the requirement for the coupling strength in Theorem~\ref{thm_freq_sync} is greater than that in the classical Kuramoto model.
Nevertheless, this is unavoidable.
Consider $N$ oscillators with $\omega_1 = 0$, $\omega_2=\cdots=\omega_N=\omega < 0$, $\theta_1(0) = \pi/2$, and $\theta_2(0)=\cdots=\theta_N(0) = 0$.
We must have $k\geq  \abs{\omega} = D(\Omega)$ to ensure frequency synchronization.\\
\item[(c)] For the SC Kuramoto model, in the parameter regime which is out of the scope of the analysis of our theorem, it is possible that the synchronized frequency could be larger than the maximal frequency of the natural frequencies of all the oscillators. See the example provided in Section~\ref{sec4.3}.
\end{enumerate}
\end{remark}

Briefly speaking,  identical oscillators achieve complete phase synchronization asymptotically if the diameter of initial phases is strictly less than $\pi$; while 
the non-identical oscillators achieve complete frequency synchronization asymptotically if the diameter of initial phases is  less than $\pi- \delta$ and the coupling strength is larger than  $D(\Omega)/\sin\delta$.
Moreover, the synchronized frequency equals the maximal natural frequency for the non-identical oscillators.

Our analysis relies partially on the diameter function, which is a common technique in the literature \parencite{ Chopra2009a, Choi2012a,Ha2014a,Hsia2020a}; see Lemma~\ref{lem_sector_trapping}.
Since the coupling function in \eqref{eq_max_kuramoto} is neither odd nor analytic,  methods based on the {\L}ojasiewicz gradient inequality, Lyapunov function \parencite{Dong2013a,Ha2013a} and the order-parameter \parencite{Benedetto2015a} cannot be applied. 
The proof of Theorem~\ref{thm_phase_sync} is based on the sector trapping property described  by Lemma~\ref{lem_sector_trapping}.
Besides employing Theorem~\ref{thm_phase_sync},  the proof of Theorem~\ref{thm_freq_sync}   relies on  a well-ordering property of  the solutions of \eqref{eq_max_kuramoto} described in  Lemma~\ref{lem_order}, a refinement of Lemma~\ref{lem_sector_trapping}, which shows that after sufficiently long time, oscillators with larger natural frequencies will be ahead in phase of those with smaller natural frequencies.

The rest part of this article  is organized as follows.  We prove Theorem~\ref{thm_phase_sync} in Section 2. In Section 3, we prove the well-ordering property and Theorem~\ref{thm_freq_sync}.  In Section 4, we demonstrate numerical experiments for  \eqref{eq_max_kuramoto}, and
we make a comparison between the classical Kuramoto model \eqref{eq_kuramoto_classical} and the SC Kuramoto model \eqref{eq_max_kuramoto}.

\section{Identical Oscillators for SC Kuramoto Model}
 In this section, we consider the SC Kuramoto model \eqref{eq_max_kuramoto} with identical oscillators and give a proof to Theorem~\ref{thm_phase_sync}.
Theorem~\ref{thm_phase_sync} demonstartes that if the initial phases are confined in a half circle, then the oscillators achieve phase synchronization.

We start with the following lemma.
It states that the oscillators will concentrate in a small region for  large coupling strength.
Similar lemmas have been used in the literature \parencite{Choi2012a,Ha2014a,Hsia2019a,Hsia2020a}.

\begin{lemma}[Sector trapping lemma]\label{lem_sector_trapping}
Assume  $k>D(\Omega)/\sin\delta$ for some $\delta\in(0,\pi/2)$. Let $\Theta(t)$ be a solution to \eqref{eq_max_kuramoto} with
$D(\Theta(0)) \leq\pi - \delta$, then $D(\Theta(t)) \leq \delta$ for
\begin{equation}\label{T0}
	t \geq T_0 : = \frac{\pi-2\delta}{k\sin\delta - D(\Omega)}.
\end{equation}
\end{lemma}

\begin{proof}
	Notice that $\pi-\delta>\delta$ and $D(\Theta(0))\leq \pi - \delta$.
	If at some time $s\geq 0$, we have $D(\Theta(s)) \in (\delta,\pi-\delta)$ for any $i, j \in \{1, 2, 3, \cdots, N \}$ such that $\theta_i(s) - \theta_j(s) = D(\Theta(s))$,
	 then we see that
	$$
	\theta_j(s) \le \theta_{\ell} (s) \le \theta_i(s), \text{ for } \ell =1,2, \cdots, N,  
	$$
	 and 
	\[
		\dot{\theta}_i(s) - \dot{\theta}_j(s)
		= \omega_i - \omega_j - k\sum_{\ell=1}^N \sin(\theta_\ell(s) - \theta_j(s))
		\leq D(\Omega) - k\sin(\theta_i(s) - \theta_j(s)).
	\]
	Since $\sin(\theta_i(s) - \theta_j(s)) > \sin\delta$ for $ \theta_i(s) - \theta_j(s) \in  (\delta,\pi-\delta) $  and  $ D(\Omega) < k \sin \delta $, we see that
	\[
		\dot{\theta}_i(s) - \dot{\theta}_j(s) < D(\Omega) - k\sin\delta < 0.
	\]
	Hence, $D(\Theta(t))$ decreases at a rate faster than $D(\Omega) - k\sin\delta$.
	This proves Lemma~\ref{lem_sector_trapping}.
\end{proof}

Now, we are in a position to prove Theorem~\ref{thm_phase_sync}.

\begin{proof}[Proof of Theorem~\ref{thm_phase_sync}]
	For any $\delta>0$,	by Lemma~\ref{lem_sector_trapping}, $D(\Theta(t))\leq\delta$ for $t\gg 0$.
	Hence, $\lim_{t\to\infty}D(\Theta(t)) = 0$.
\end{proof}

\section{Non-identical Oscillators  for SC Kuramoto Model}
In this section, we analyze  the the SC Kuramoto model for  non-identical oscillators.
Theorem~\ref{thm_freq_sync} shows that if the initial phases are confined in a half circle and the coupling strength is  large, then the oscillators achieve frequency synchronization asymptotically. 

For a real-valued function $f$ defined on an open set $U\subseteq\R$, we define its right derivative (if the limit exists) by
\[
	D_+f(x) = \lim_{h\downarrow 0}\frac{f(x+h) - f(x)}{h}, \quad\forall x\in U.
\]

The following proposition will be used in this section.
The proof of the proposition is omitted since it is straightforward.
\begin{proposition}\label{prop_max_derivative}
	Let $U\subseteq\R$ be an open set and $f_1,\ldots,f_n:U\to\R$ be continuous functions having right derivatives.
	Define $F(x) = \max_{1 \le i \le n}f_i(x)$.
	Then, we have
	\[
		D_+F(x) = \max_{i\in I_x} D_+f_i(x),\quad\forall x\in U,
	\]
	where $I_x \coloneqq \{ 1 \le i \le n  \mid F(x) = f_i(x)\}$.
\end{proposition}

As mentioned in the introduction, the proof of Theorem~\ref{thm_freq_sync} relies on Lemma~\ref{lem_order}, a well-ordering property of the solutions of \eqref{eq_max_kuramoto}.
To prove Lemma~\ref{lem_order}, we prepare the following lemma.

\begin{lemma}\label{lem_interchange}
		Let $\Theta(t)$ be a solution to \eqref{eq_max_kuramoto}.
		Assume, for some $t_0 \ge 0$, $D(\Theta(t))\leq\pi/2$ for all $t\geq t_0$.
		If $\omega_i > \omega_j$, then $$\theta_i(t) \geq \theta_j(t)$$ for
		\[
			t \geq t_0 + \max \left\{0, \frac{ \theta_j(t_0) - \theta_i(t_0)  } {\omega_i - \omega_j} \right\}.
		\]
\end{lemma}
\begin{proof}

\textbf{Claim:} If  $\theta_i(t) - \theta_j(t) > 0$ at some moment $t_1 \ge t_0$, then $\theta_i(t) - \theta_j(t) > 0$ for all $t \ge t_1$.

We shall use proof by contradiction to verify the claim.  Suppose the claim does not hold. Let $t= t_2 > t_1$ be the first moment such that $\theta_i(t) = \theta_j(t),$ i.e.,
\begin{equation}
\label{eqFirst}
\theta_i(t) - \theta_j(t) > 0 \,\,  \text{ for }  \,\,  t \in [ t_1, t_2) \,\,  \text{ and } \,\, \theta_i(t_2) - \theta_j(t_2) = 0.
\end{equation}
Inferring from \eqref{eqFirst}, we have 
\begin{equation}
\dot{\theta}_i(t_2) - \dot{\theta}_j(t_2)  \le 0.
\end{equation}
However, taking the difference of the $i$-th and $j$-th equations of \eqref{eq_max_kuramoto} at $t = t_2$, we obtain
$$\dot{\theta}_i(t_2) - \dot{\theta}_j(t_2) = \omega_i - \omega_j >0,$$
which violates \eqref{eqFirst}. This proves the claim.

	Next, if $\theta_i(t) - \theta_j(t) \leq 0$ for some $t\geq t_0$, then
	\[
		\dot{\theta}_j(t) - \dot{\theta}_i(t) = (\omega_j-\omega_i) + k\sum_{\ell=1}^N \Big( \max\{0,\sin(\theta_\ell(t) - \theta_j(t))\} - \max\{0,\sin(\theta_\ell(t) - \theta_i(t))\} \Big).
	\]
	Since
	\[
		\frac{\pi}{2} \geq \theta_\ell(t) - \theta_i(t) \geq \theta_\ell(t) - \theta_j(t) \geq -\frac{\pi}{2},
	\]
	we have
	$$\max\{0,\sin(\theta_\ell(t) - \theta_j(t))\} - \max\{0,\sin(\theta_\ell(t) - \theta_i(t))\} \leq 0.$$
	Hence,
	 $$\dot{\theta}_j(t) - \dot{\theta}_i(t) \leq \omega_j - \omega_i < 0.$$
	This means that the difference of $\theta_j$ and $\theta_i$ decreases at a rate faster than $\omega_i-\omega_j$.
	The lemma follows.
\end{proof}

\begin{lemma}[Well-ordering Lemma]\label{lem_order}
Under the assumptions of Theorem~\ref{thm_freq_sync},	let $\Theta(t)$ be a solution to \eqref{eq_max_kuramoto}
	 with  $D(\Theta(0)) \leq \pi - \delta$. Then there exists $T_*\geq 0$ such that, for all $t > T_*$, we have
	 \begin{align}	
&D(\Theta(t))\leq \delta,  \text{ and }   \\
&\theta_i(t)\geq \theta_j(t)  \text{ if }  \omega_i > \omega_j.
 	 \end{align}
\end{lemma}

\begin{proof}
Let $ T_0 $ be as defined in \eqref{T0}. By Lemma~\ref{lem_sector_trapping}, we see  that $D(\Theta(t))\leq\delta<\pi/2$ for all $t\ge T_0$.
	Applying Lemma~\ref{lem_interchange} to every pair $(\theta_i(t),\theta_j(t))$ with $\omega_i > \omega_j$,  we have $\theta_i(t)\geq\theta_j(t)$ for
	\[
		t\geq  T_* : = T_0 + \max_{(i,j):\omega_i>\omega_j} \max\Bigg\{0,\frac{\theta_j(T_0) - \theta_i(T_0)}{\omega_i - \omega_j}\Bigg\}.\qedhere
	\]
\end{proof}
We are now in a position  to prove Theorem~\ref{thm_freq_sync}.

\begin{proof}[Proof of Theorem~\ref{thm_freq_sync}]
	Assume, among $\Omega$, there are $M$ distinct natural frequencies $$ \omega_1= \omega_{j_1}>\omega_{j_2}>\cdots>\omega_{j_M} = \omega_N.$$
	Without the loss of generality, we may assume $\omega_1=0$.
	Partition $N$ oscillators into $M$ groups $\Theta^{(1)},\ldots,\Theta^{(M)}$, where all oscillators in the group $\Theta^{(m)}$ have the same natural frequency $\omega_{j_m}$.
	Write
	\[
		\Theta^{(m)} = (\theta^{(m)}_1, \ldots, \theta^{(m)}_{n_m} ),\quad m = 1, 2, 3, \cdots, M,
	\]
	and $\sum_{m=1}^M n_m = N$, where $n_m$ is the number of oscillators with natural frequency $\omega_{j_m}$.

	By Lemma~\ref{lem_order}, there exists $T_*>0$ such that for any $t\geq T_*$, $m\in \{1, 2, \cdots, M-1\}$, $i_1 \in \{1,2, \cdots, n_m\}$ and  $i_2 \in \{1,2, \cdots, n_{m+1}\}$, we have
	\begin{align}
		&D(\Theta(t))\leq \delta < \pi / 2, \label{eq_group_bounded}\\
		&\theta^{(m)}_{i_1}(t) \geq \theta^{(m+1)}_{i_2}(t). \label{eq_group_order}
	\end{align}
	Therefore,  for $m\in \{1, 2, \cdots, M\}$ and $i\in \{1, 2, \cdots, n_m\}$ and $t \ge T_*$, \eqref{eq_max_kuramoto} reads
	\begin{equation}
	\label{eq_group_order_kuramoto}
	\begin{aligned}
		\dot{\theta}^{(m)}_{i}(t)
		=& \omega_{j_m} + k\sum_{m'=1}^{m-1}\sum_{\ell=1}^{n_{m'}} \sin(\theta_\ell^{(m')}(t) - \theta_i^{(m)}(t))\\
	&+ k\sum_{\ell=1}^{n_m} \max\Big\{0, \sin(\theta_\ell^{(m)}(t) - \theta_i^{(m)}(t))\Big\},
		\end{aligned}
	\end{equation}
	 In the following, we prove that the frequency of each oscillator in $\Theta^{(m)}$ will converge to $0$ by proceeding a mathematical induction argument on $m$.

	\paragraph{Base case.}
	For $m=1$, \eqref{eq_group_order_kuramoto} reads
	\begin{equation}\label{eq_group1}
	\begin{split}
		\dot{\theta}^{(1)}_{i}(t)
		&= \omega_1 + k\sum_{\ell=1}^{n_1} \Big\{0, \sin(\theta_\ell^{(1)}(t) - \theta_i^{(1)}(t))\Big\}, \quad i =  1, 2, \cdots, n_1.
		\end{split}
	\end{equation}
	Since $k>0$, by Theorem~\ref{thm_phase_sync}, $\Theta^{(1)}(t)$ shall achieve a complete phase synchronization asymptotically.
	By passing limit  $ t \to \infty$ in \eqref{eq_group1}, we conclude that $\smash{\lim_{t\to\infty}\dot{\theta}_i^{(1)}(t) = \omega_1 = 0}$ for $i=1,2, \cdots, n_1$.
	This proves the base case.

\paragraph{Induction Step.}
	Let $2\leq m\leq M$.
	Assume $\lim_{t\to\infty}\dot{\theta}^{(m')}_{\ell}(t) = 0$ for all $1 \le m' \le m-1$ and $1 \le \ell \le n_{m'}$.
	Fix any $\varepsilon>0$.
	By the above induction assumption, there exists $t_1\geq T_*$ such that for all $t\geq t_1$, $1 \le m' \le m-1$ and $1 \le \ell \le n_{m'}$, we have
	\begin{equation}\label{eq_small_group_freq}
		\abs{\dot{\theta}_\ell^{(m')}(t)} \leq \frac{\varepsilon}{2}.
	\end{equation}
	Define
	\[
	r_j(t) = \dot{\theta}_j^{(m)}(t),\ M(t) = \max_{1\le j \le  n_m} r_j(t),\text{ and }m(t) = \min_{1 \le j \le n_m} r_j(t).
	\]
	We shall show that $\lim_{t\to\infty} M(t) = \lim_{t\to\infty} m(t) = 0$, which implies $\lim_{t\to\infty} \dot{\theta}_j^{(m)}(t)=0$ for $j =1, 2, \cdots, n_m$.
	
	Assume $M(s) \geq \varepsilon$ for some $s\geq t_1$.
	Let $i\in I_s \coloneqq \{ 1 \le j \le n_m \mid r_j(s) = M(s)\}$.
	By \eqref{eq_group_order_kuramoto} and Proposition~\ref{prop_max_derivative}, $r_i(t)$ is a continuous function, and its right derivative is given by
	\[
		D_+r_i(s)
		= k\sum_{m'=1}^{m-1}\sum_{\ell=1}^{n_{m'}} \cos(\theta_\ell^{(m')}(s) - \theta_i^{(m)}(t))(\dot{\theta}_\ell^{(m')}(s) - r_i(s))
		+ k\sum_{\ell=1}^{n_m} \Delta_{i\ell}(s),
	\]
	where
	\begin{align*}
		&\Delta_{i\ell}(s) \coloneqq D_+\max\Big\{0, \sin(\theta_\ell^{(m)}(s) - \theta_i^{(m)}(s))\Big\} \\
		&= \begin{cases}
			\cos(\theta_\ell^{(m)}(s) - \theta_i^{(m)}(s))(r_\ell(s) - r_i(s))
			&\qif \sin(\theta_\ell^{(m)}(s) - \theta_i^{(m)}(s)) > 0 \\
			\max\Big\{0, \cos(\theta_\ell^{(m)}(s) - \theta_i^{(m)}(s))(r_\ell(s) - r_i(s)) \Big\}
			&\qif \sin(\theta_\ell^{(m)}(s) - \theta_i^{(m)}(s)) = 0 \\
			0 &\qif \sin(\theta_\ell^{(m)}(s) - \theta_i^{(m)}(s)) < 0.
		\end{cases}
	\end{align*}
Inferring from   \eqref{eq_group_bounded}, we see that $\smash{\cos(\theta_\ell^{(m)}(s) - \theta_i^{(m)}(s)) \ge \cos \delta > 0}$.  On the other hand, by the choice of $i$, we have $r_i(s) = M(s) \geq r_\ell(s)$ for $\ell =1, 2, \cdots, n_m$. We then conclude  $\Delta_{i\ell}(s) \leq 0$.
	By \eqref{eq_small_group_freq} and the assumption $M(s) \ge \varepsilon$, we have
	\[
		-\frac{3\varepsilon}{2} \leq \dot{\theta}_\ell^{(m')}(s) - r_i(s) \leq -\frac{\varepsilon}{2}.
	\]
	Therefore,
	\[
		\begin{split}
		D_+r_i(s)
		&\leq k\sum_{m'=1}^{m-1}\sum_{\ell=1}^{n_{m'}} \cos(\theta_\ell^{(m')}(s) - \theta_i^{(m)}(s))(\dot{\theta}_\ell^{(m')}(s) - r_i(s)) \\
		&\leq -\frac{\varepsilon k}{2}(n_1 + \cdots + n_{m-1})(\cos \delta).
		\end{split}
	\]
	Hence, by Proposition~\ref{prop_max_derivative},
	\[
		D_+M(s)
		= \max_{i\in I_s}D_+r_i(s)
		\leq -\frac{\varepsilon k}{2}(n_1 + \cdots + n_{m-1})(\cos \delta) < 0.
	\]
	This shows that if $M(t)\geq\varepsilon$, then $M(t)$ decays at a rate faster than
	\[
		-\frac{\varepsilon k}{2}(n_1 + \cdots + n_{m-1})(\cos \delta).
	\]
	Therefore, $M(t) \leq \varepsilon$ for $t\gg 0$.
	By a similar argument, one can show that $m(t) \geq -\varepsilon$ for $t\gg 0$.
	Since this holds for any $\varepsilon>0$ and $M(t)\geq m(t)$ for all $t\geq 0$, we have $\lim_{t\to\infty}M(t) = \lim_{t\to\infty}m(t) = 0$.
	This proves the induction step.
	The theorem then follows from the mathematical induction.
\end{proof}


\section{Numerical Results}
In this section, we present numerical results for  \eqref{eq_max_kuramoto} and compare the synchronization behavior of the SC Kuramoto model \eqref{eq_max_kuramoto} with that of the classical Kuramoto model \eqref{eq_kuramoto_classical}.
Note that most authors write $k/N$ for $k$ in the classical Kuramoto model \eqref{eq_kuramoto_classical}.
In order to make the comparison,  we  do not follow this convention.

In this section, all the differential equations were solved numerically by using the \texttt{solve\_ivp()} function in the SciPy package with $10^{-5}$ relative tolerance.
The natural frequencies and the initial phases were first uniformly generated from $[0,1]$ and then scaled to satisfy the given diameter.

\subsection{Identical Oscillators}

\begin{figure}[t]
\centering
\begin{subfigure}{.5\textwidth}
	\centering
	\includegraphics[width=\textwidth]{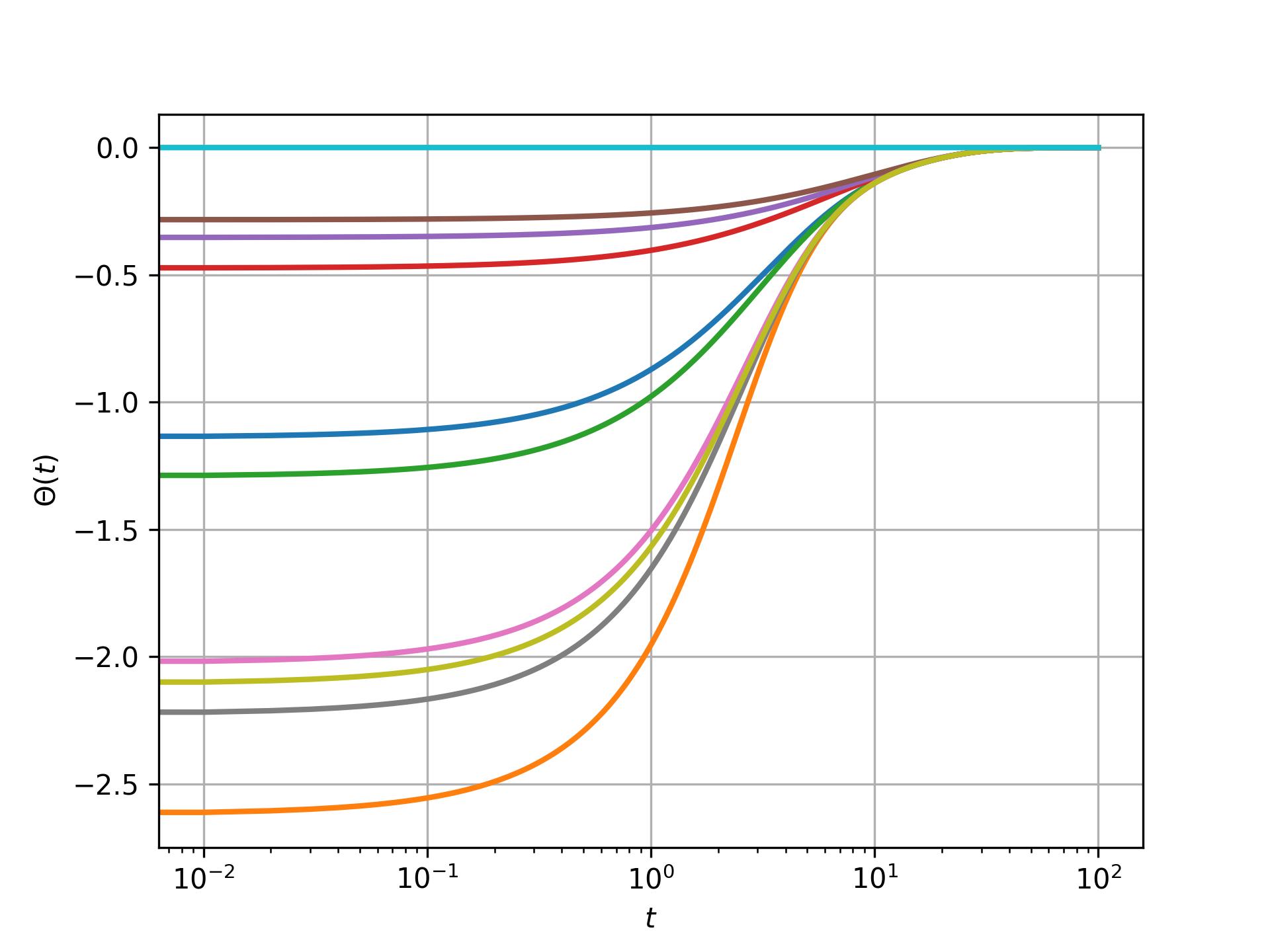}
	\caption{$\Theta(t)$.}
\end{subfigure}%
\hfill
\begin{subfigure}{.5\textwidth}
	\centering
	\includegraphics[width=\textwidth]{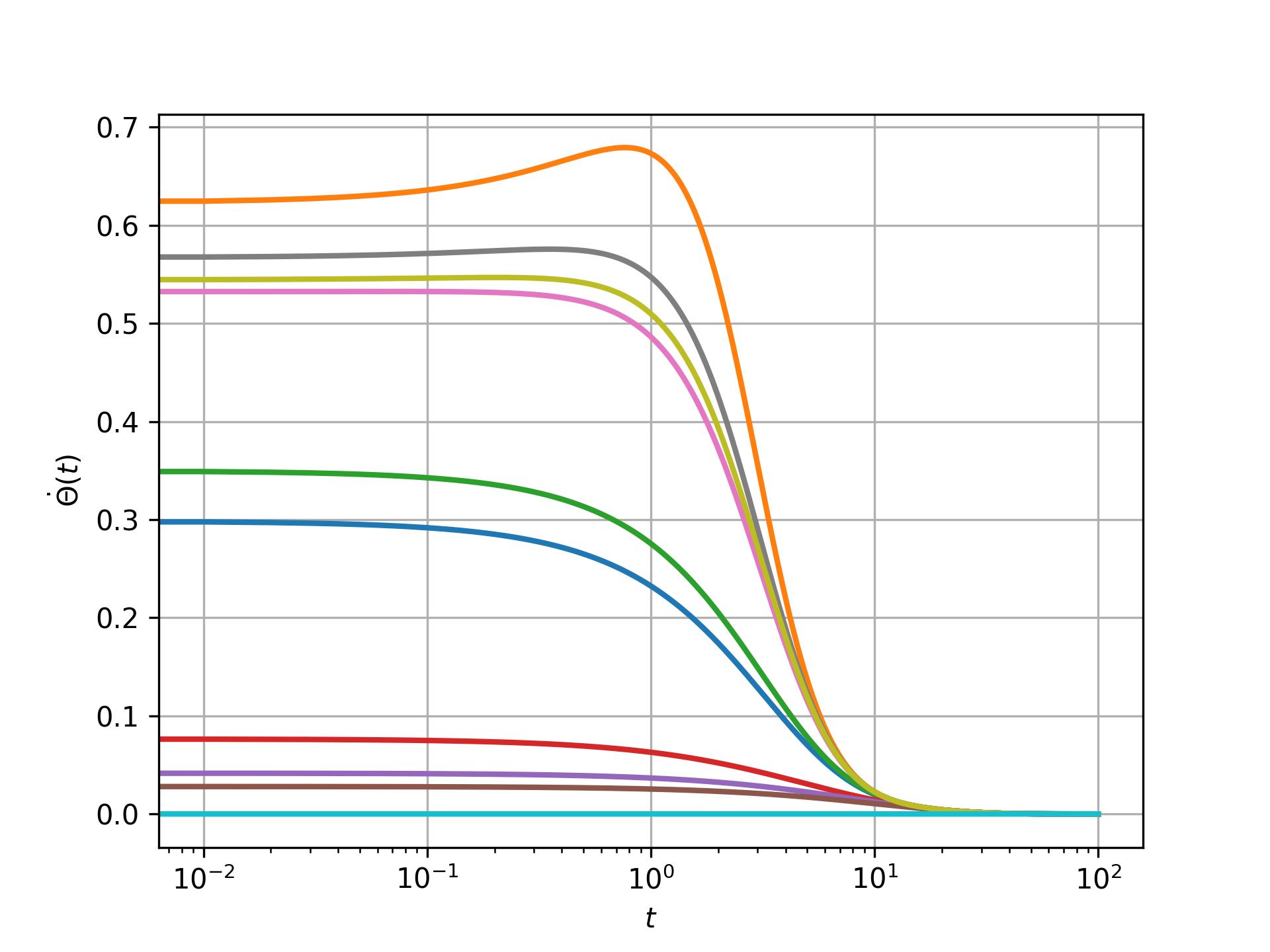}
	\caption{$\dot{\Theta}(t)$.}
\end{subfigure}%

\begin{subfigure}{.5\textwidth}
	\centering
	\includegraphics[width=\textwidth]{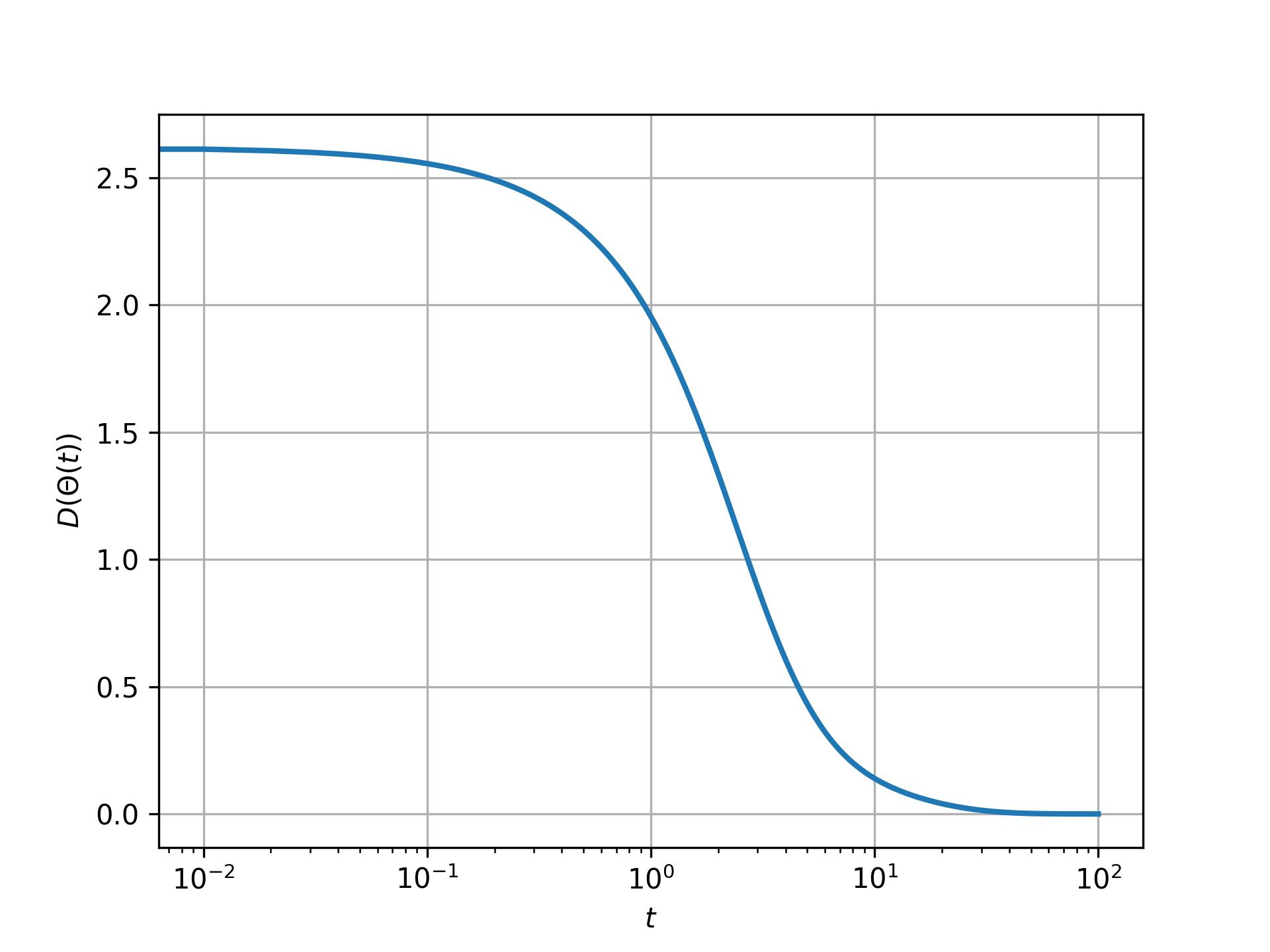}
	\caption{$D(\Theta(t))$}
\end{subfigure}%
\hfill
\begin{subfigure}{.5\textwidth}
	\centering
	\includegraphics[width=\textwidth]{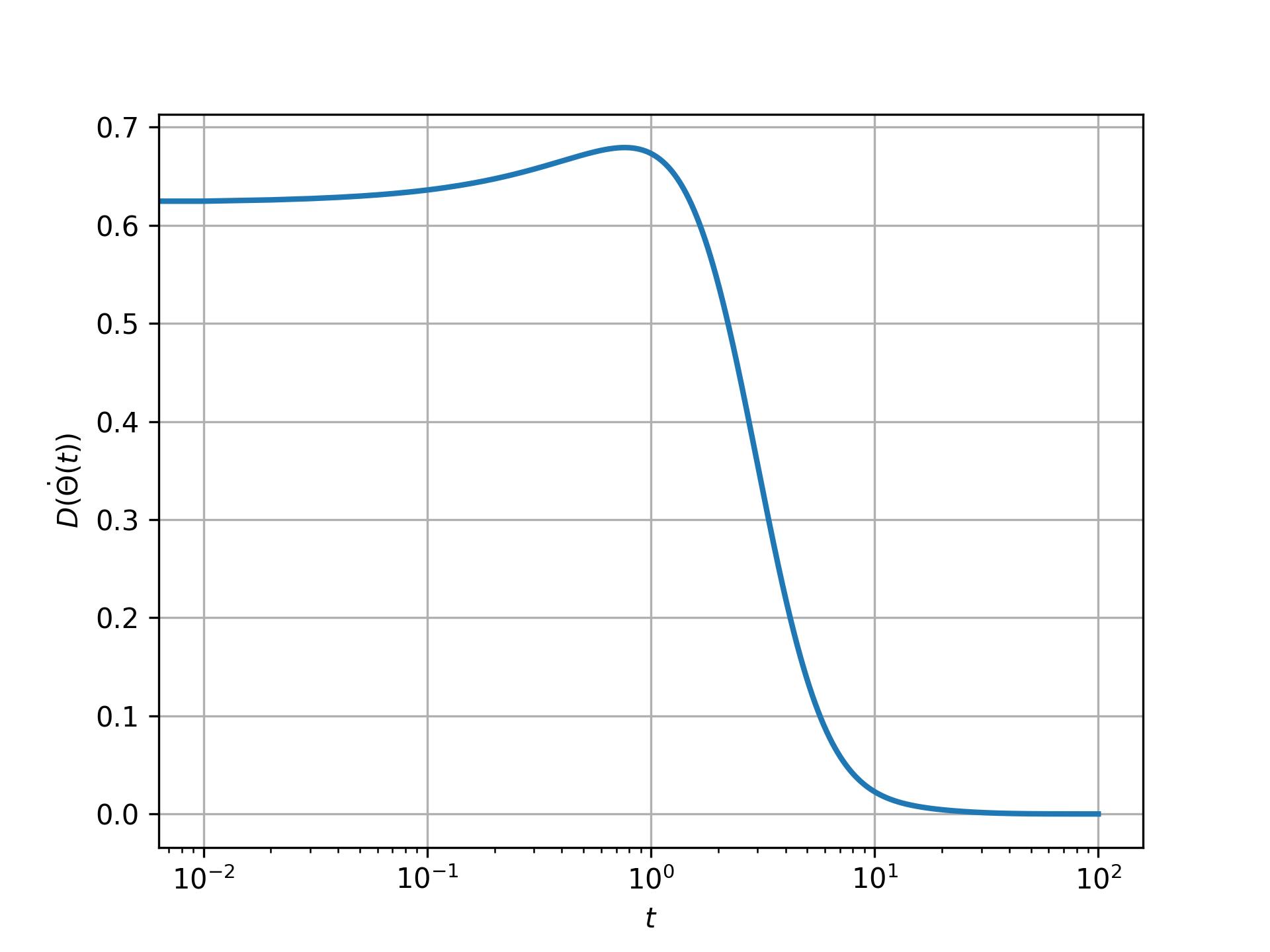}
	\caption{$D(\dot{\Theta}(t))$}
\end{subfigure}%

\caption{Solution of the SC Kuramoto model \eqref{eq_max_kuramoto} with $N=10$, $k=0.1$, $D(\Theta(0))=5\pi/6$, and $D(\Omega)=0$.}
\label{fig:identical}
\end{figure}

In Theorem~\ref{thm_phase_sync}, we have shown that for the SC Kuramoto model \eqref{eq_max_kuramoto}, if all oscillators are identical, $D(\Theta(0)) < \pi$, and $k>0$, then they achieve complete phase synchronization asymptotically.
We demonstrate Theorem~\ref{thm_phase_sync} numerically in Figure~\ref{fig:identical}.
We consider $10$ oscillators and set the coupling strength $k$ to be $0.1$.
Their natural frequencies are $0$, and their initial phases are confined in a $5\pi/6$ arc.
This initial condition satisfies the assumptions of Theorem~\ref{thm_phase_sync}.
As suggested by the theorem, the oscillators achieve complete phase synchronization.

Recall that for the classical Kuramoto model \eqref{eq_kuramoto_classical}, if all oscillators are identical and $k>0$, then the oscillators with $D(\Theta(0)) < \pi$ achieve complete phase synchronization asymptotically.
In Figure~\ref{fig:identical_compare}, we compare the convergence speed of the SC Kuramoto model \eqref{eq_max_kuramoto} with that of the classical Kuramoto model \eqref{eq_kuramoto_classical}.
We observe that in terms of both phase and frequency, the classical Kuramoto oscillators converge faster than the SC Kuramoto oscillators.

\begin{figure}[t]
\centering
\begin{subfigure}{.5\textwidth}
	\centering
	\includegraphics[width=\textwidth]{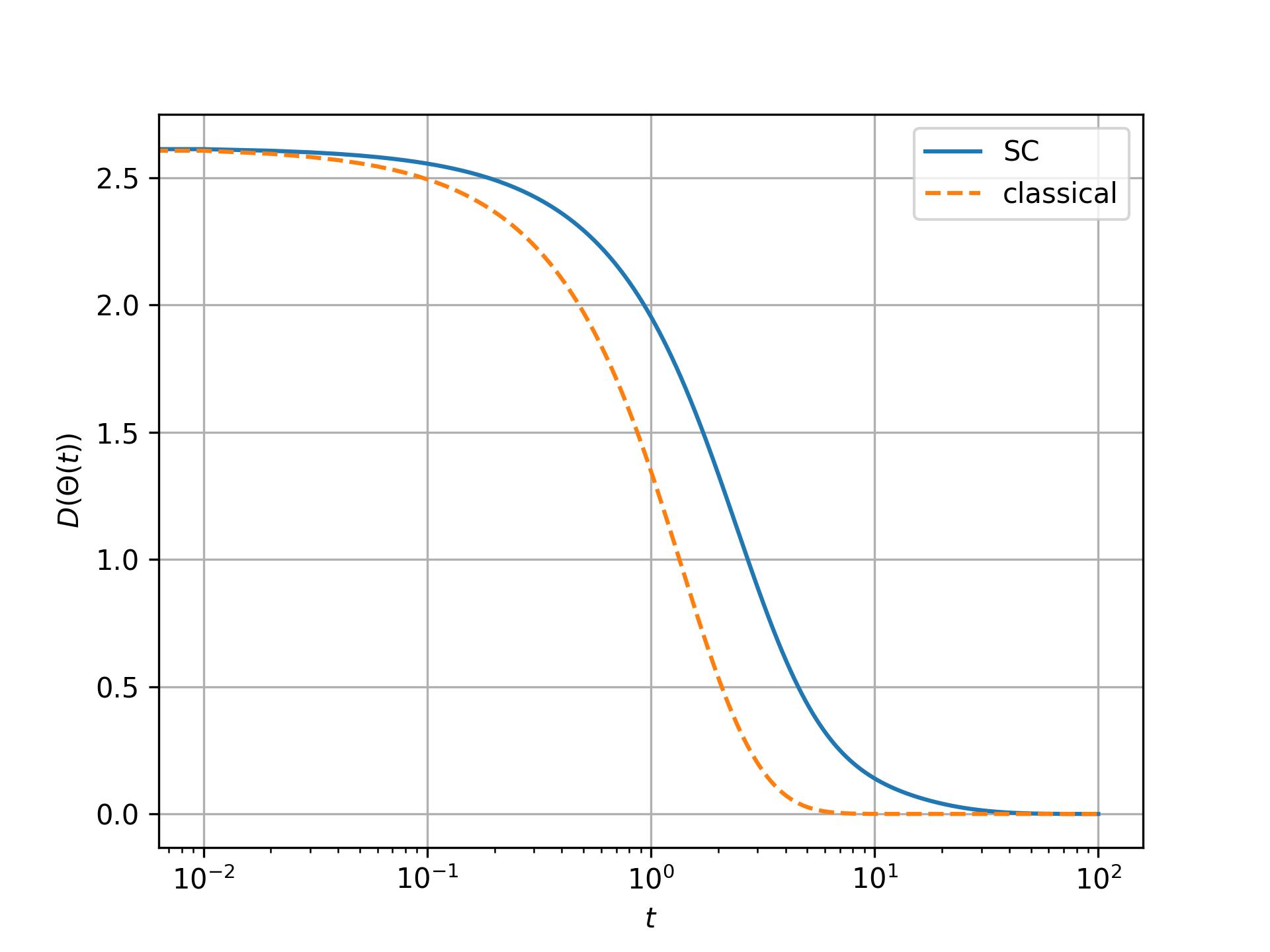}
	\caption{$D(\Theta(t))$.}
\end{subfigure}%
\hfill
\begin{subfigure}{.5\textwidth}
	\centering
	\includegraphics[width=\textwidth]{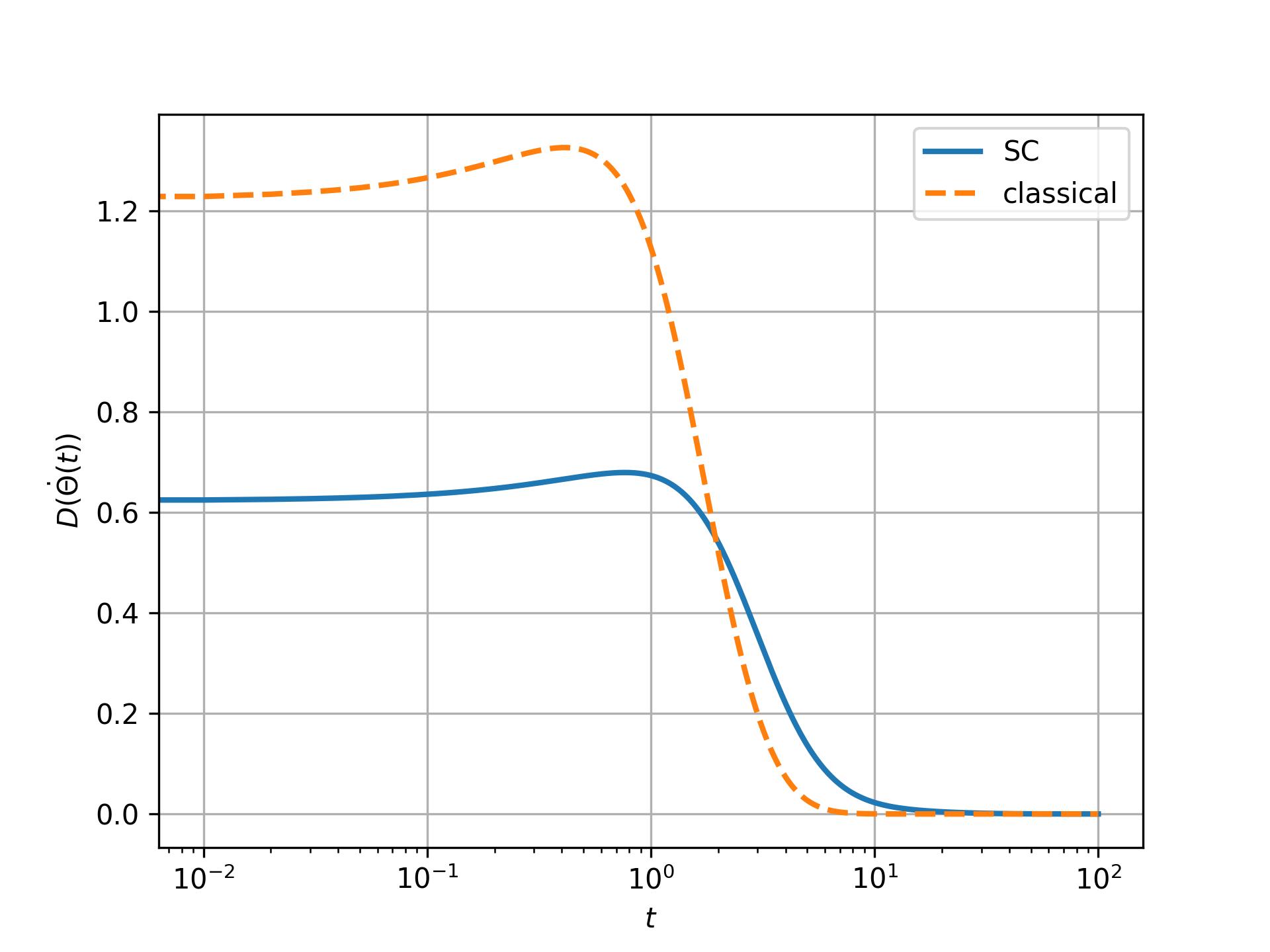}
	\caption{$D(\dot{\Theta}(t))$}
\end{subfigure}%
\caption{Comparison of the convergence speeds of the SC Kuramoto model \eqref{eq_max_kuramoto} and the classical Kuramoto model \eqref{eq_kuramoto_classical}.
We set $N=10$, $k=0.1$, $D(\Theta(0))=5\pi/6$, and $D(\Omega)=0$.}
\label{fig:identical_compare}
\end{figure}

\subsection{Non-identical Oscillators}
In this subsection, we consider non-identical oscillators.
Theorem~\ref{thm_freq_sync} shows that for the SC Kuramoto model \eqref{eq_max_kuramoto}, if   $k>D(\Omega)/\sin\delta$, then the oscillators with $D(\Theta(0)) < \pi-\delta$ achieve a complete frequency synchronization asymptotically.
In Figure~\ref{fig:nonidentical}, we consider $10$ oscillators with $D(\Omega) = 1$ and $\max\Omega=0$.
We set the coupling strength $k$ to be $D(\Omega)/\sin(\pi/6)+10^{-3}$.
Their initial phases are confined in a $5\pi/6-10^{-3}$ arc.
This initial condition satisfies the assumptions of Theorem~\ref{thm_freq_sync}.
We observe that the oscillators achieve a complete frequency synchronization asymptotically and  the synchronized frequency equals the largest natural frequency, which is consistent with Theorem~\ref{thm_freq_sync}.

\begin{figure}[t]
\centering
\begin{subfigure}{.5\textwidth}
	\centering
	\includegraphics[width=\textwidth]{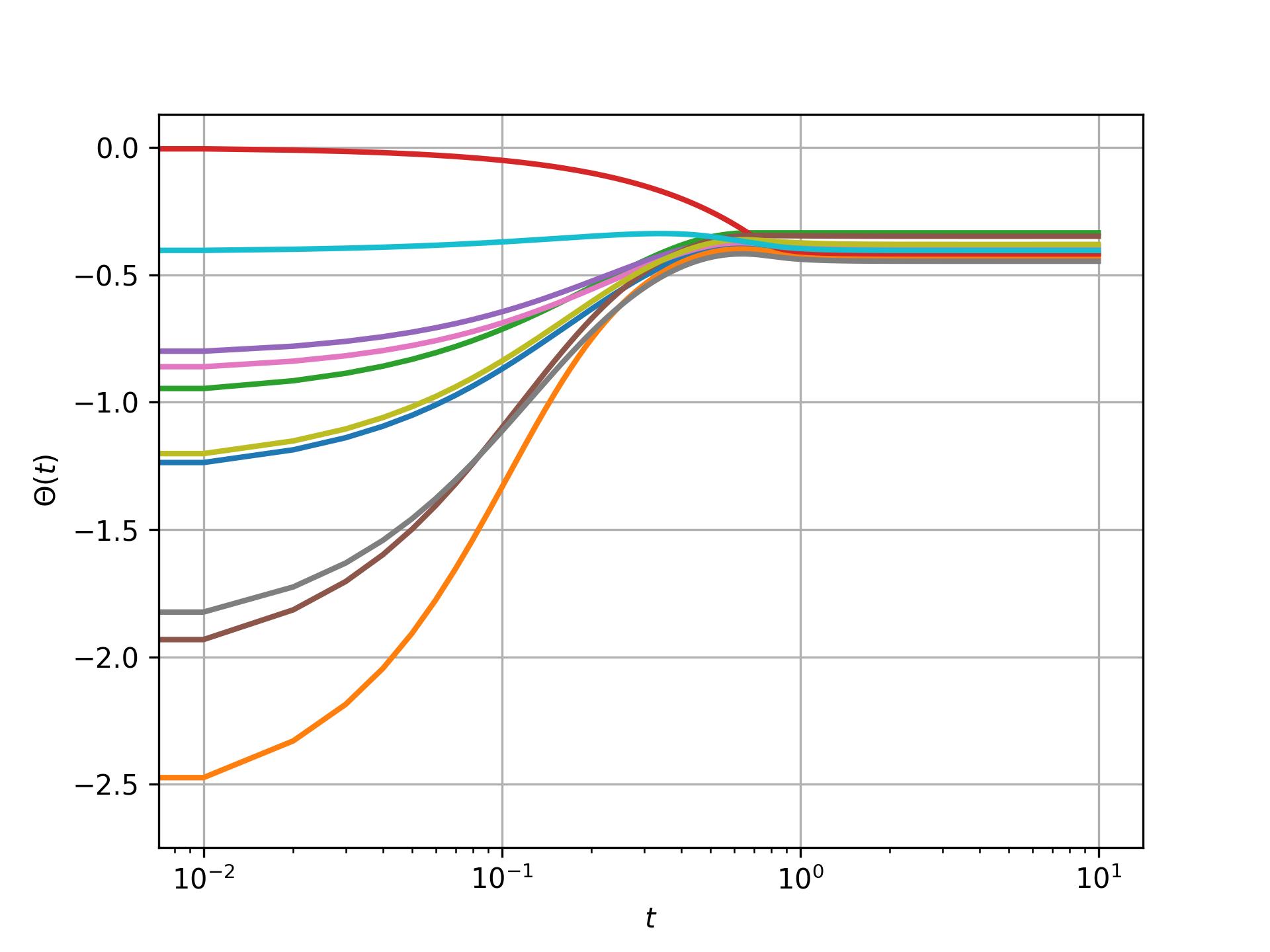}
	\caption{$\Theta(t)$.}
\end{subfigure}%
\hfill
\begin{subfigure}{.5\textwidth}
	\centering
	\includegraphics[width=\textwidth]{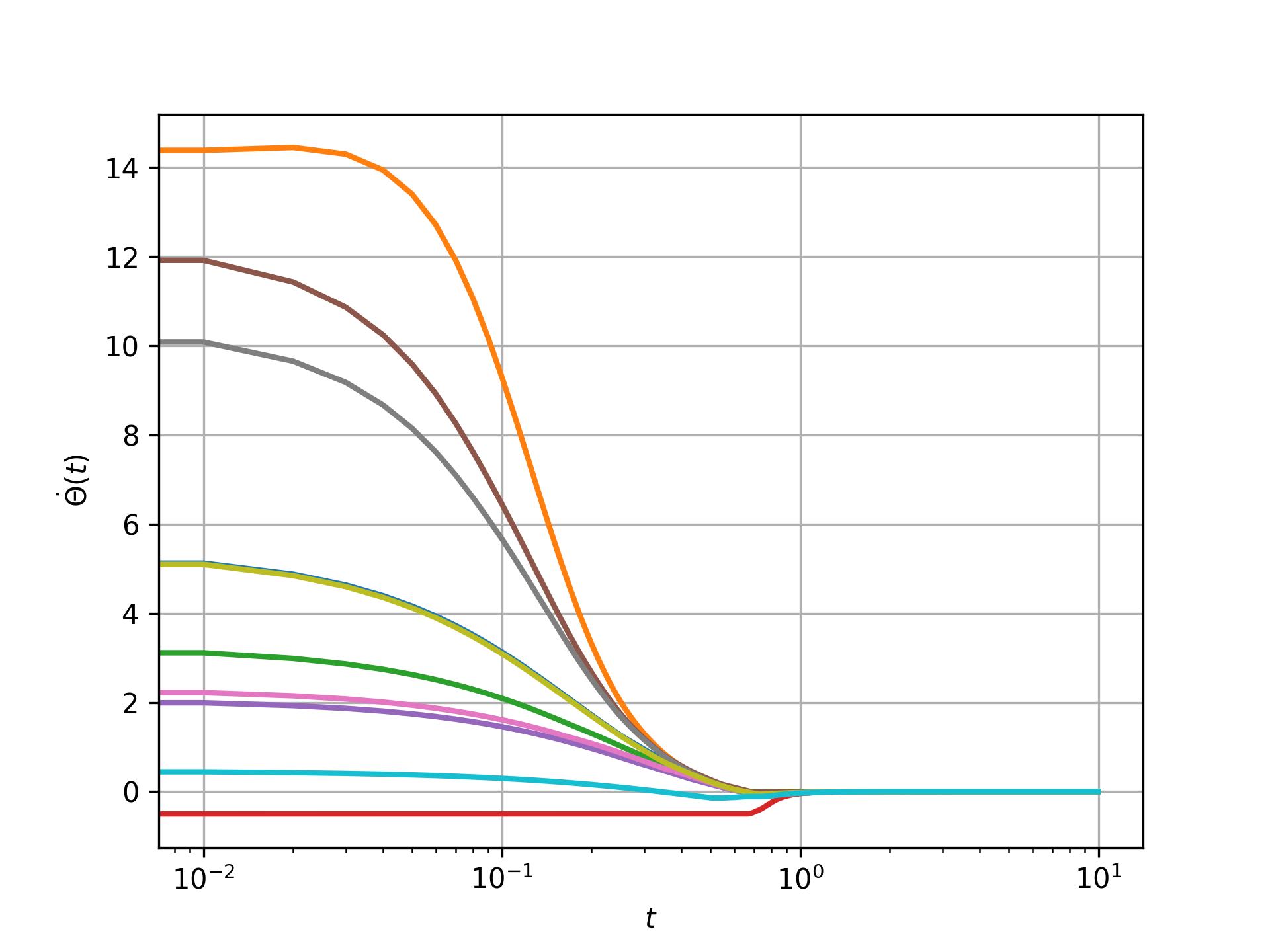}
	\caption{$\dot{\Theta}(t)$.}
\end{subfigure}%

\begin{subfigure}{.5\textwidth}
	\centering
	\includegraphics[width=\textwidth]{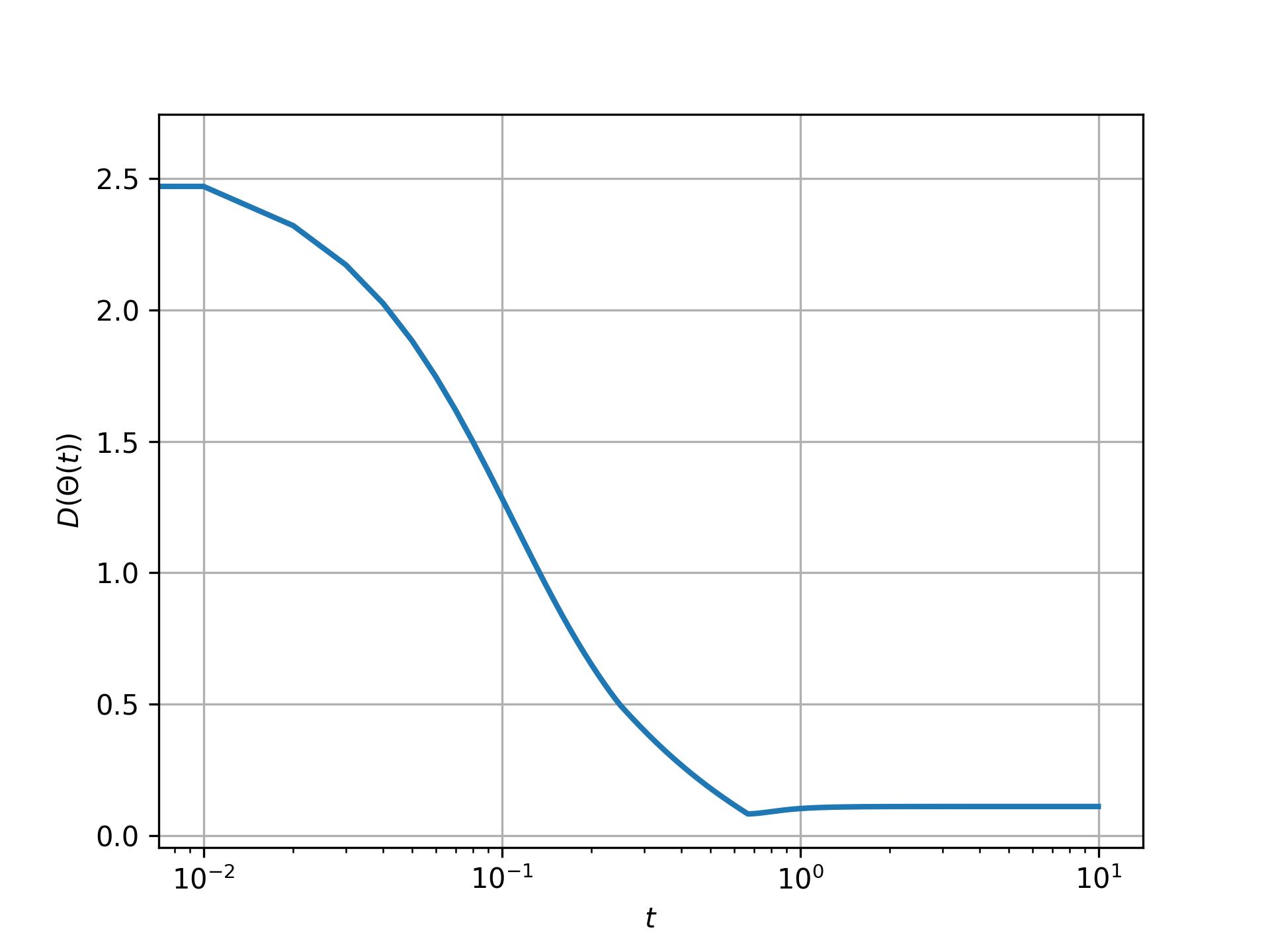}
	\caption{$D(\Theta(t))$}
\end{subfigure}%
\hfill
\begin{subfigure}{.5\textwidth}
	\centering
	\includegraphics[width=\textwidth]{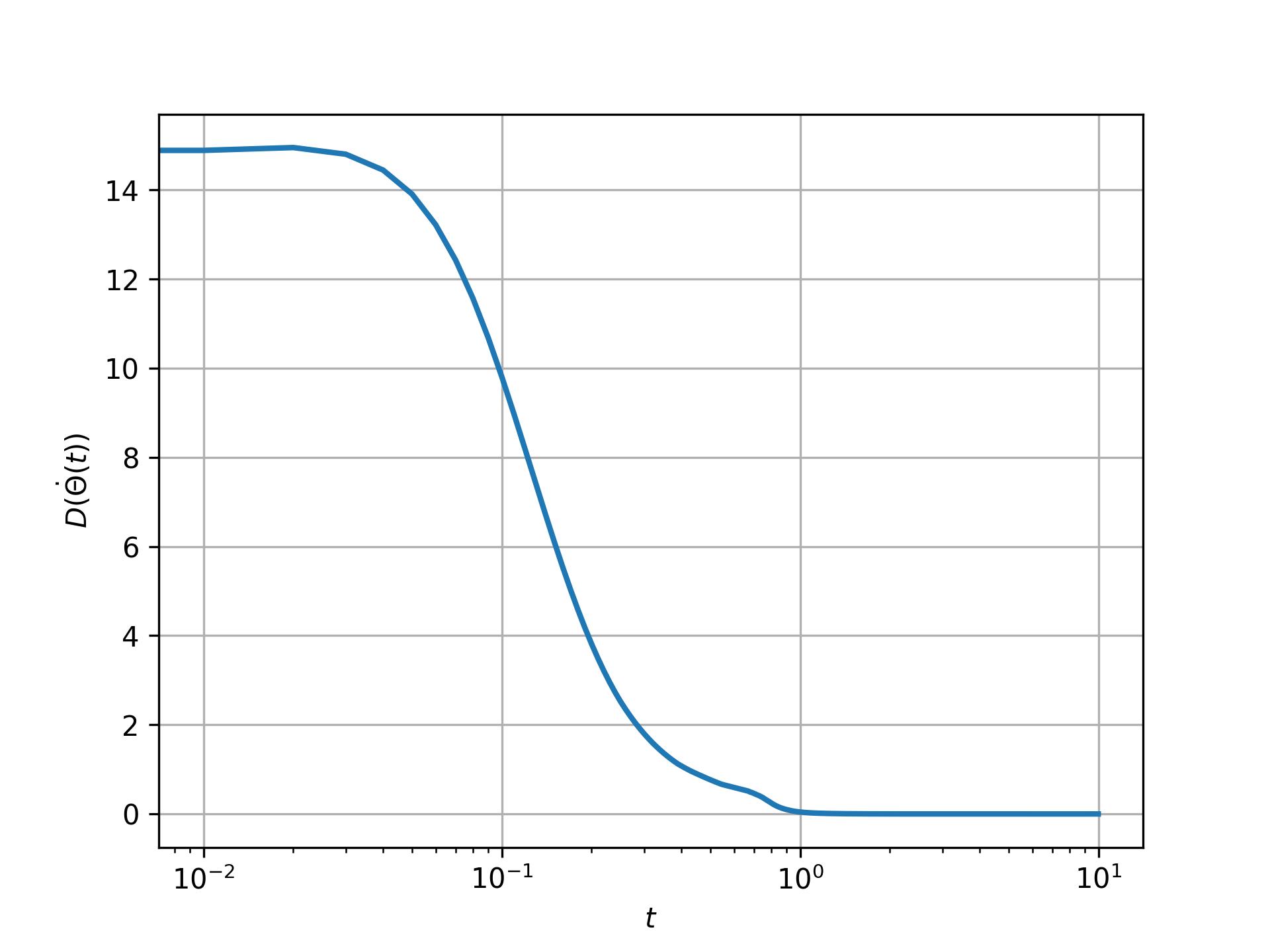}
	\caption{$D(\dot{\Theta}(t))$}
\end{subfigure}%

\caption{Solution of the SC Kuramoto model \eqref{eq_max_kuramoto} with $N=10$, $k= D(\Omega)/\sin(\pi/6)+10^{-3} \approx 1.967$, $D(\Theta(0)) = 5\pi/6 - 10^{-3} \approx 2.617$, $D(\Omega)=1$, and $\max\Omega=0$.}
\label{fig:nonidentical}
\end{figure}

Note that under the  assumption   $k> D(\Omega)/(N\sin\delta)$, the condition  $D(\Theta(0))<\pi-\delta$ suffices to ensure complete frequency synchronization for the classical Kuramoto model \eqref{eq_kuramoto_classical}.
 Since $D(\Omega)/\sin\delta > D(\Omega)/(N\sin\delta),$ this hints that the SC Kuramoto model may be harder to achieve frequency synchronization than the classical Kuramoto  model.
In Figure~\ref{fig:nonidentical_compare}, we compare the convergence speed of the SC Kuramoto model \eqref{eq_max_kuramoto} with that of the classical Kuramoto model \eqref{eq_kuramoto_classical} when both models are guaranteed to achieve complete frequency synchronization.
We observe that the classical Kuramoto oscillators achieve complete frequency synchronization faster than the SC Kuramoto oscillators.

\begin{figure}[H]
\centering
\begin{subfigure}{.5\textwidth}
	\centering
	\includegraphics[width=\textwidth]{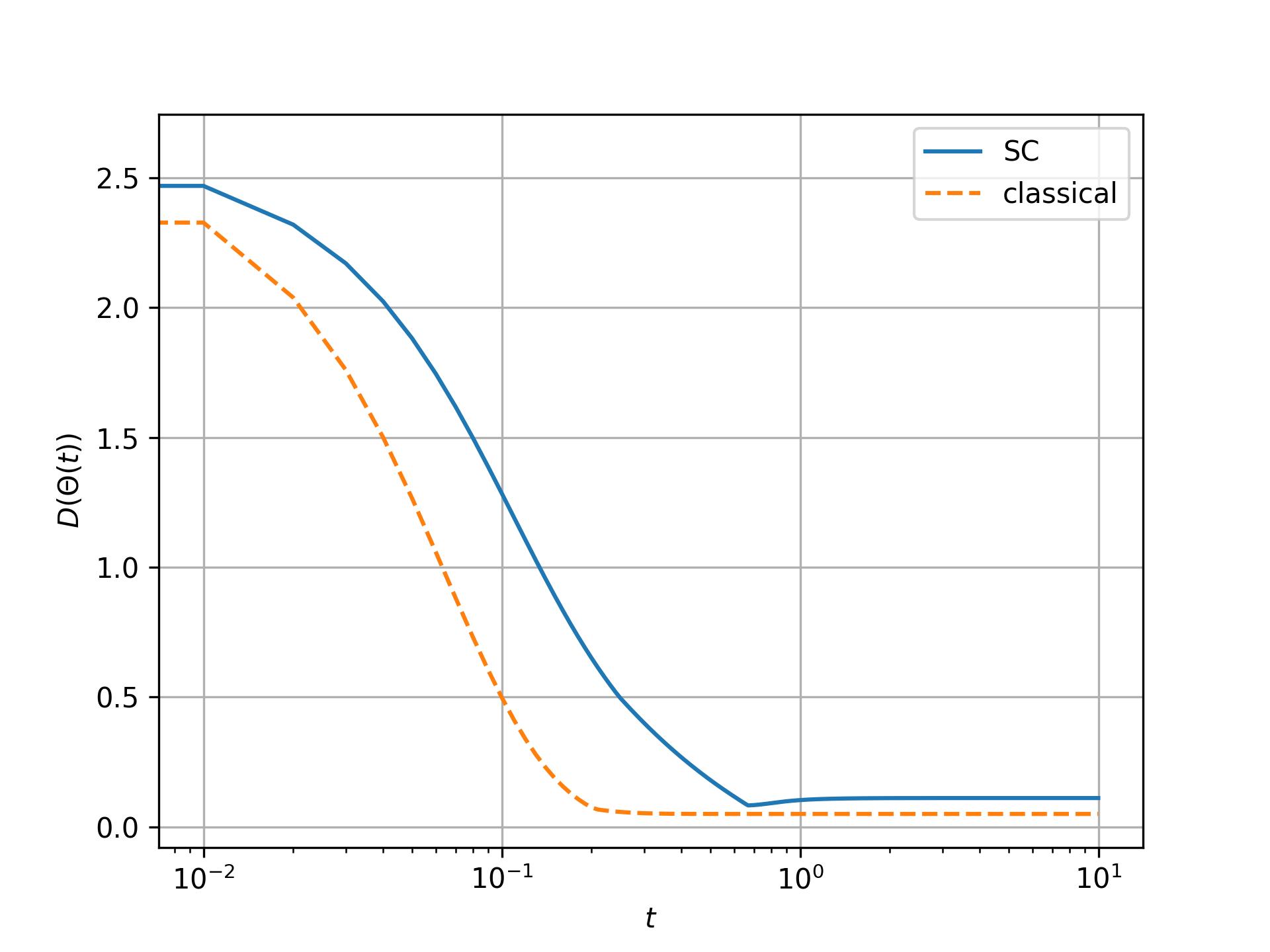}
	\caption{$D(\Theta(t))$.}
\end{subfigure}%
\hfill
\begin{subfigure}{.5\textwidth}
	\centering
	\includegraphics[width=\textwidth]{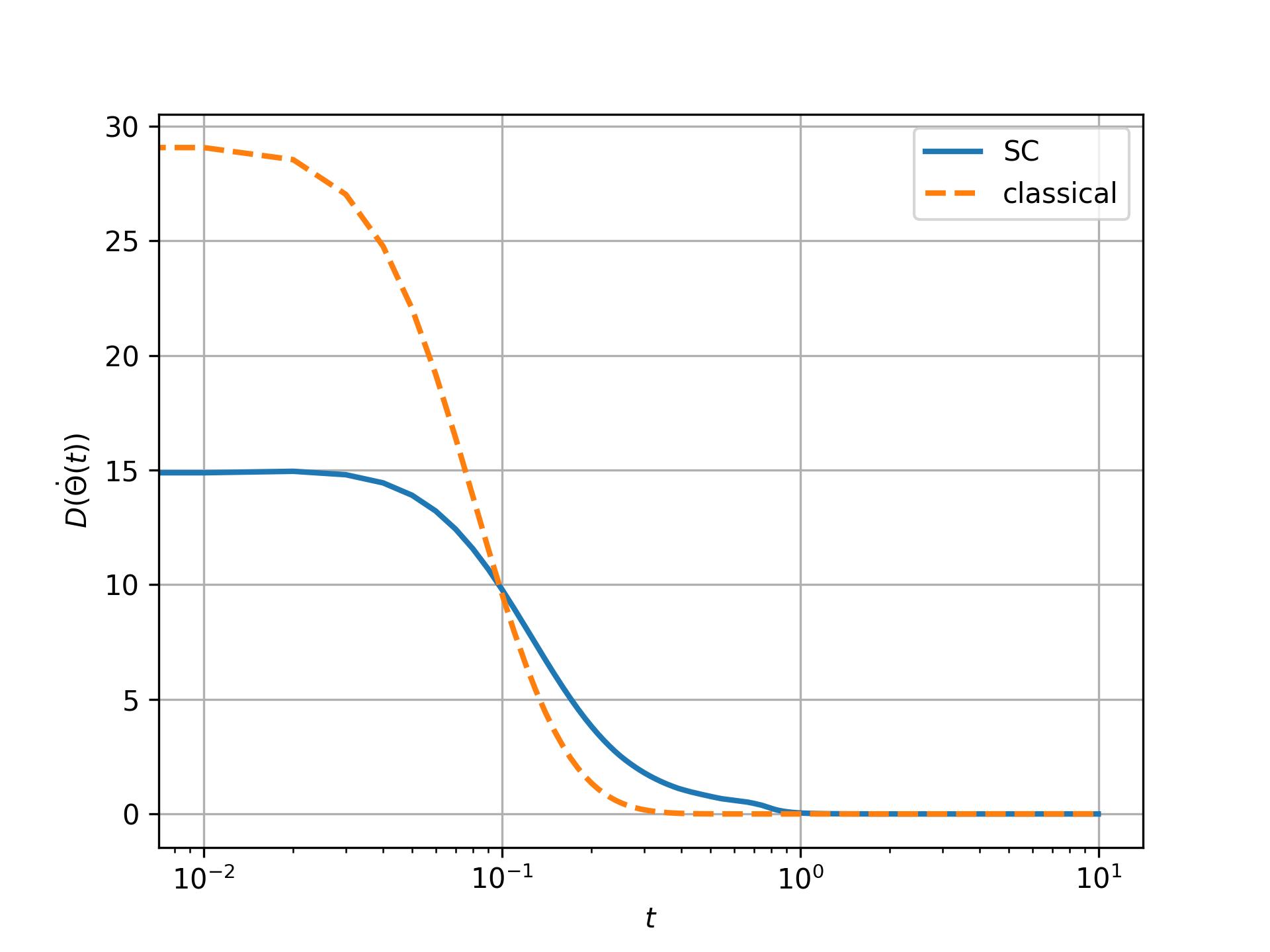}
	\caption{$D(\dot{\Theta}(t))$}
\end{subfigure}%
\caption{Comparison of the convergence speeds of the SC Kuramoto model \eqref{eq_max_kuramoto} and the classical Kuramoto model \eqref{eq_kuramoto_classical}.
We set $N=10$, $k= D(\Omega)/\sin(\pi/6)+10^{-3} \approx 1.967$, $D(\Theta(0)) = 5\pi/6 - 10^{-3} \approx 2.617$, $D(\Omega)=1$, and $\max\Omega=0$.}
\label{fig:nonidentical_compare}
\end{figure}

\begin{figure}[H]
\centering
\begin{subfigure}{.5\textwidth}
	\centering
	\includegraphics[width=\textwidth]{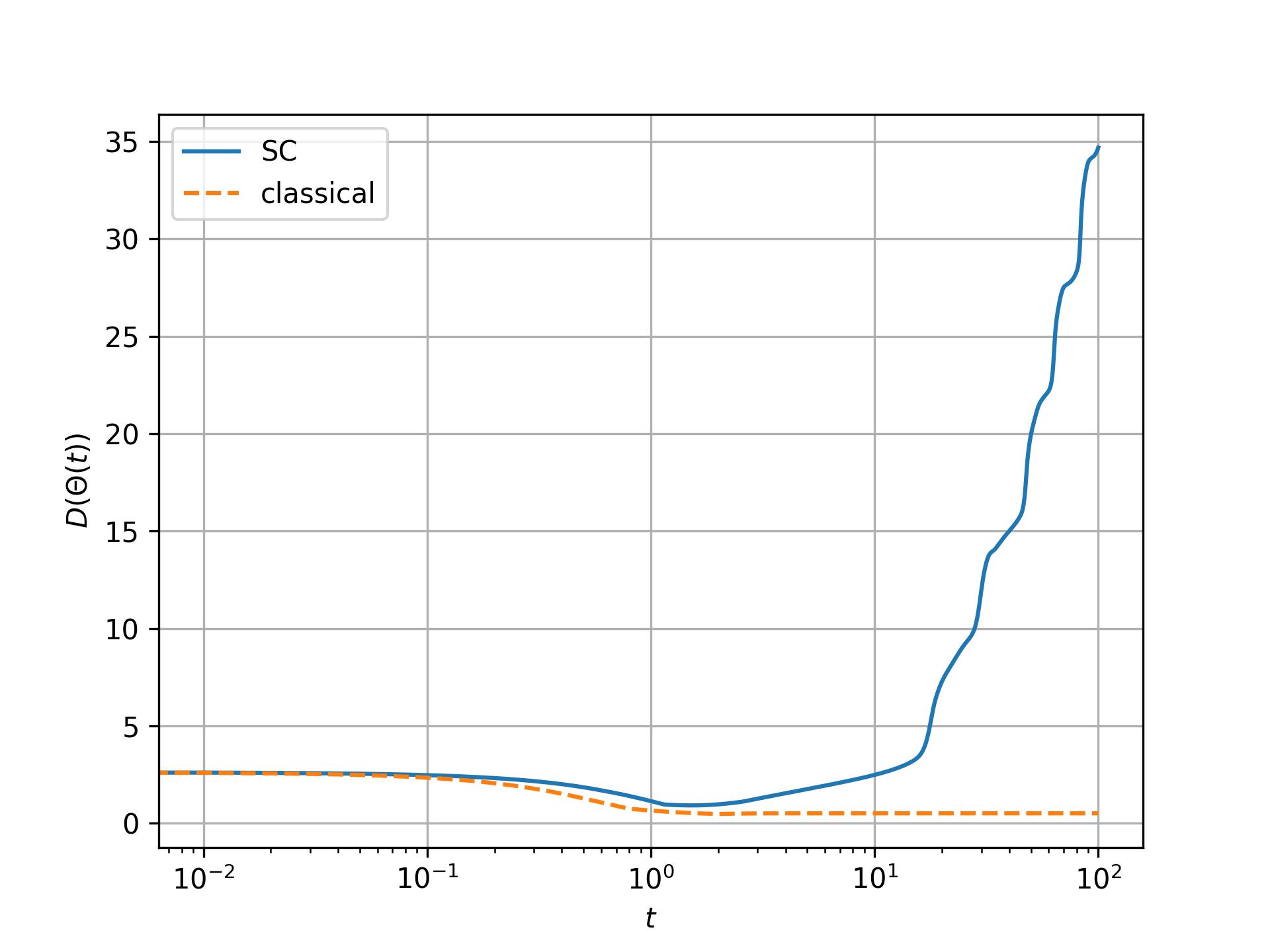}
	\caption{$D(\Theta(t))$.}
\end{subfigure}%
\hfill
\begin{subfigure}{.5\textwidth}
	\centering
	\includegraphics[width=\textwidth]{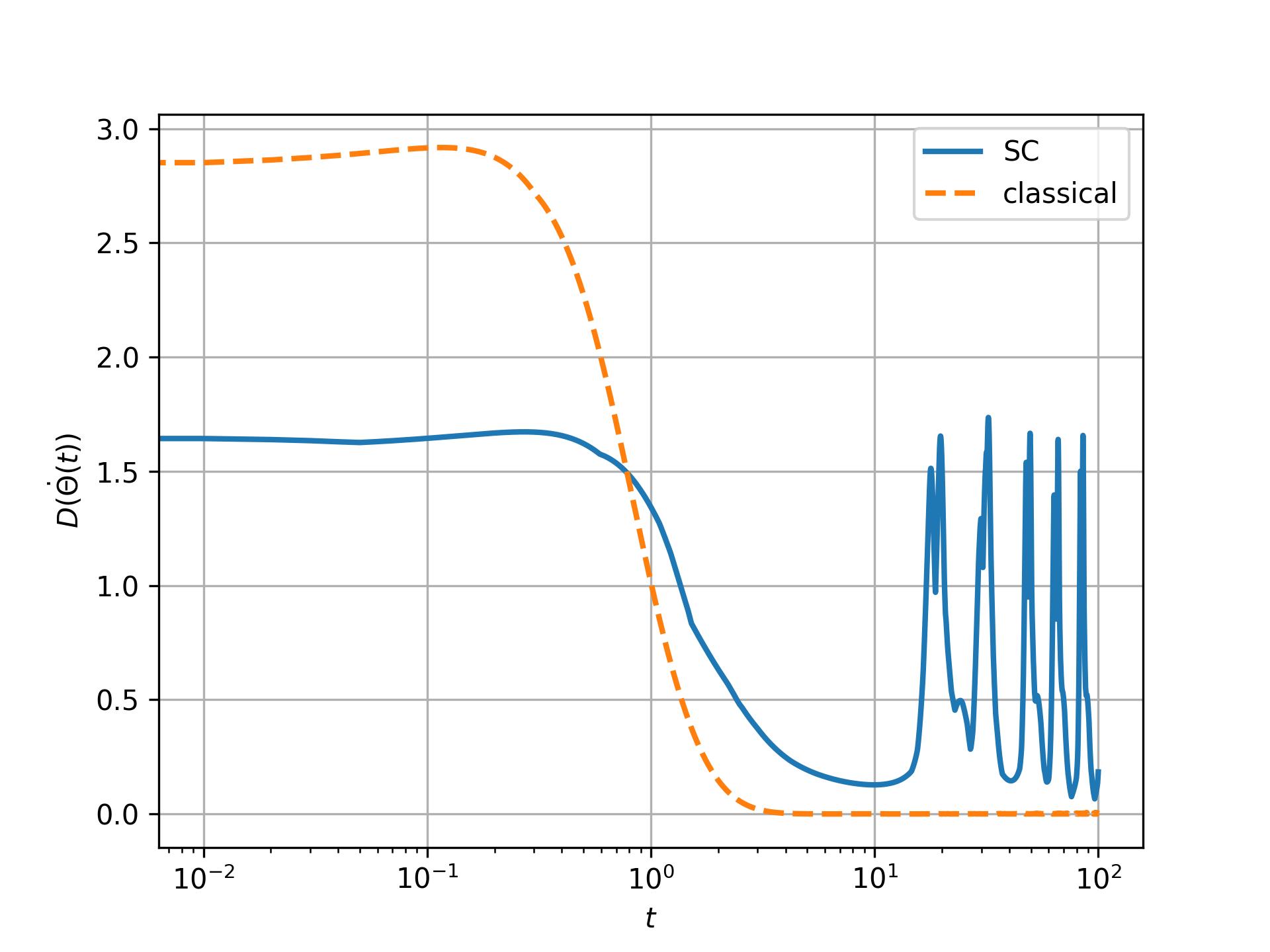}
	\caption{$D(\dot{\Theta}(t))$}
\end{subfigure}%
\caption{Comparison of the SC Kuramoto model \eqref{eq_max_kuramoto} and the classical Kuramoto model \eqref{eq_kuramoto_classical}.
We set $N=10$, $k=0.2$, $D(\Theta(0)) = 5\pi/6 - 10^{-3} \approx 2.617$, $D(\Omega)=1$, and $\max\Omega=0$.}
\label{fig:nonidentical_diverge}
\end{figure}

In Figure~\ref{fig:nonidentical_diverge}, we set $k = 0.2$ with other parameters unchanged.
Since $k > D(\Omega)/(N\sin\delta)$, the standard model is guaranteed to achieve complete frequency synchronization.
However, since $k < D(\Omega)/\sin\delta$, our theorem does not apply.
We observe that under this setup, the classical Kuramoto oscillators achieve complete frequency synchronization asymptotically, but the SC Kuramoto oscillators do not.
This confirms that, compared with  the classical Kuramoto oscillators,  it is harder for SC Kuramoto oscillators to synchronize.

\subsection{Numerical Results beyond Our Theorems}\label{sec4.3}
In this subsection, we provide numerical results that cannot be inferred from our theoretical results.
Firstly, we consider identical oscillators with $D(\Theta(0))$ larger than $\pi$.
Specifically, we set $N=10$, $k=1$, $D(\Theta(0)) = 15\pi/8>\pi$, and $D(\Omega)=0$.
Figure~\ref{fig:identical_more_than_pi} presents the numerical results.
Note that the oscillators achieve both complete phase and frequency synchronization.
Also, the solution satisfies $\lim_{t\to\infty}D(\Theta(t))=2\pi$ instead of $\lim_{t\to\infty}D(\Theta(t))=0$.

\begin{figure}[t]
\centering
\begin{subfigure}{.5\textwidth}
	\centering
	\includegraphics[width=\textwidth]{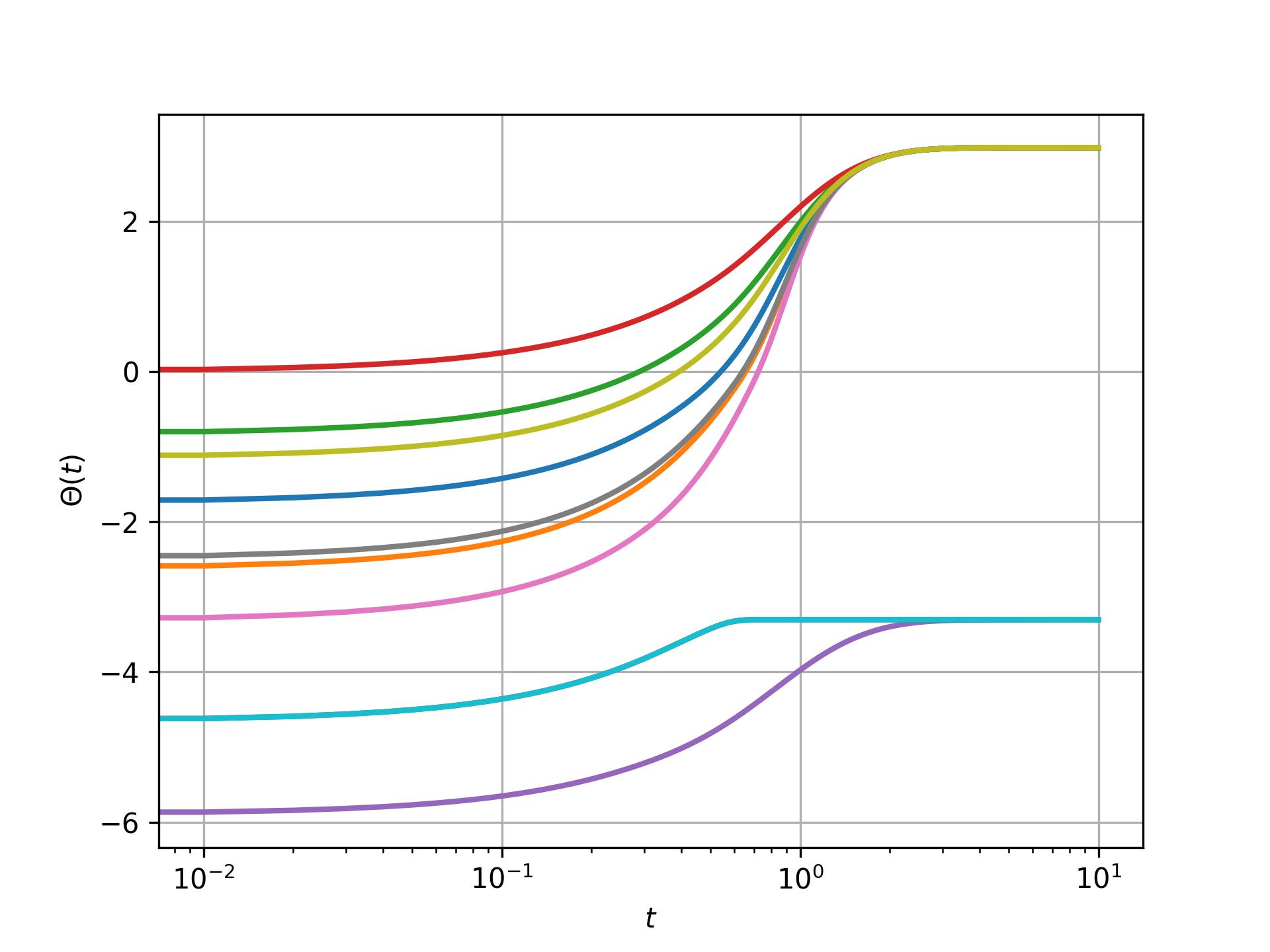}
	\caption{$\Theta(t)$.}
\end{subfigure}%
\hfill
\begin{subfigure}{.5\textwidth}
	\centering
	\includegraphics[width=\textwidth]{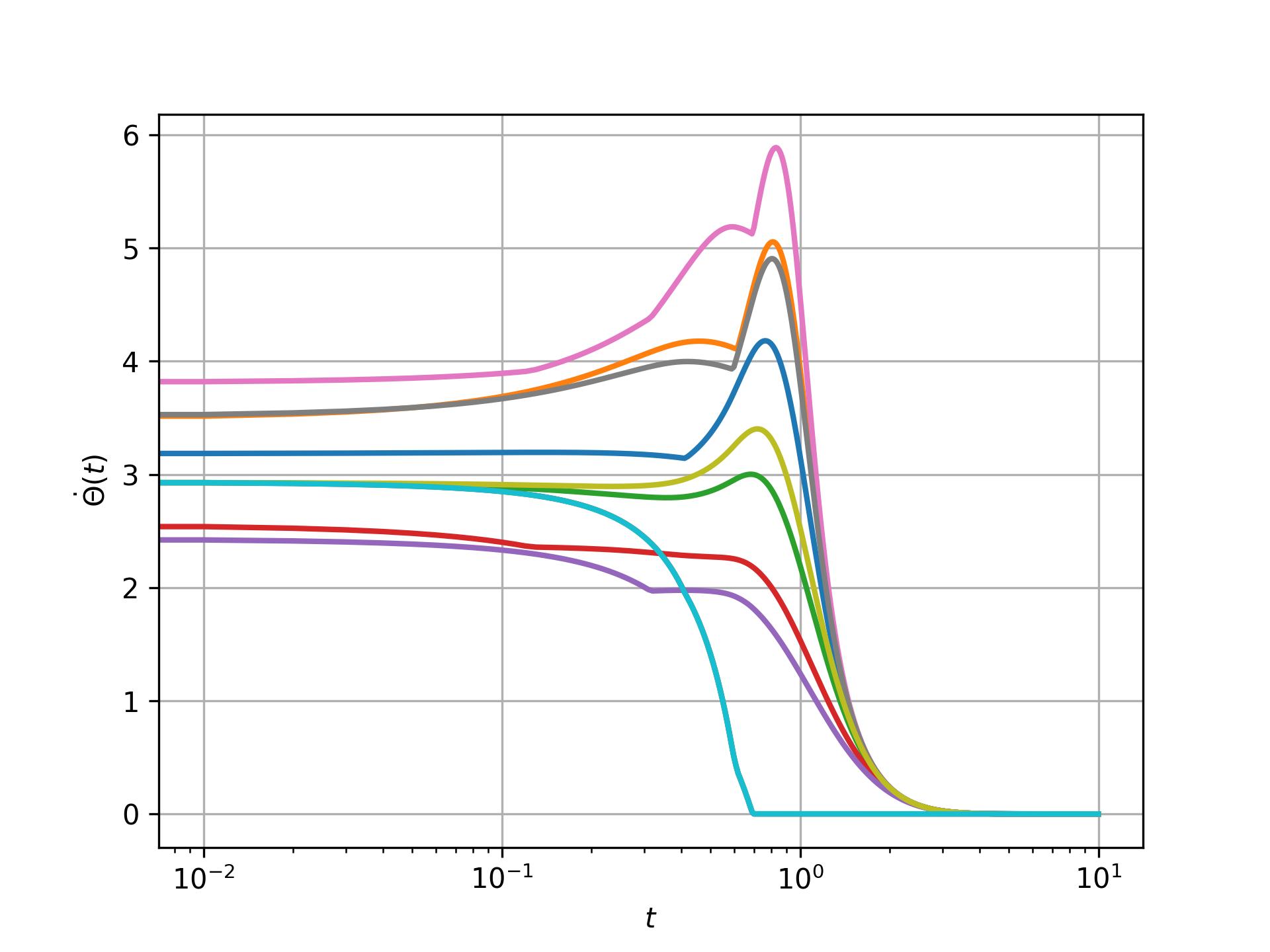}
	\caption{$\dot{\Theta}(t)$.}
\end{subfigure}%

\begin{subfigure}{.5\textwidth}
	\centering
	\includegraphics[width=\textwidth]{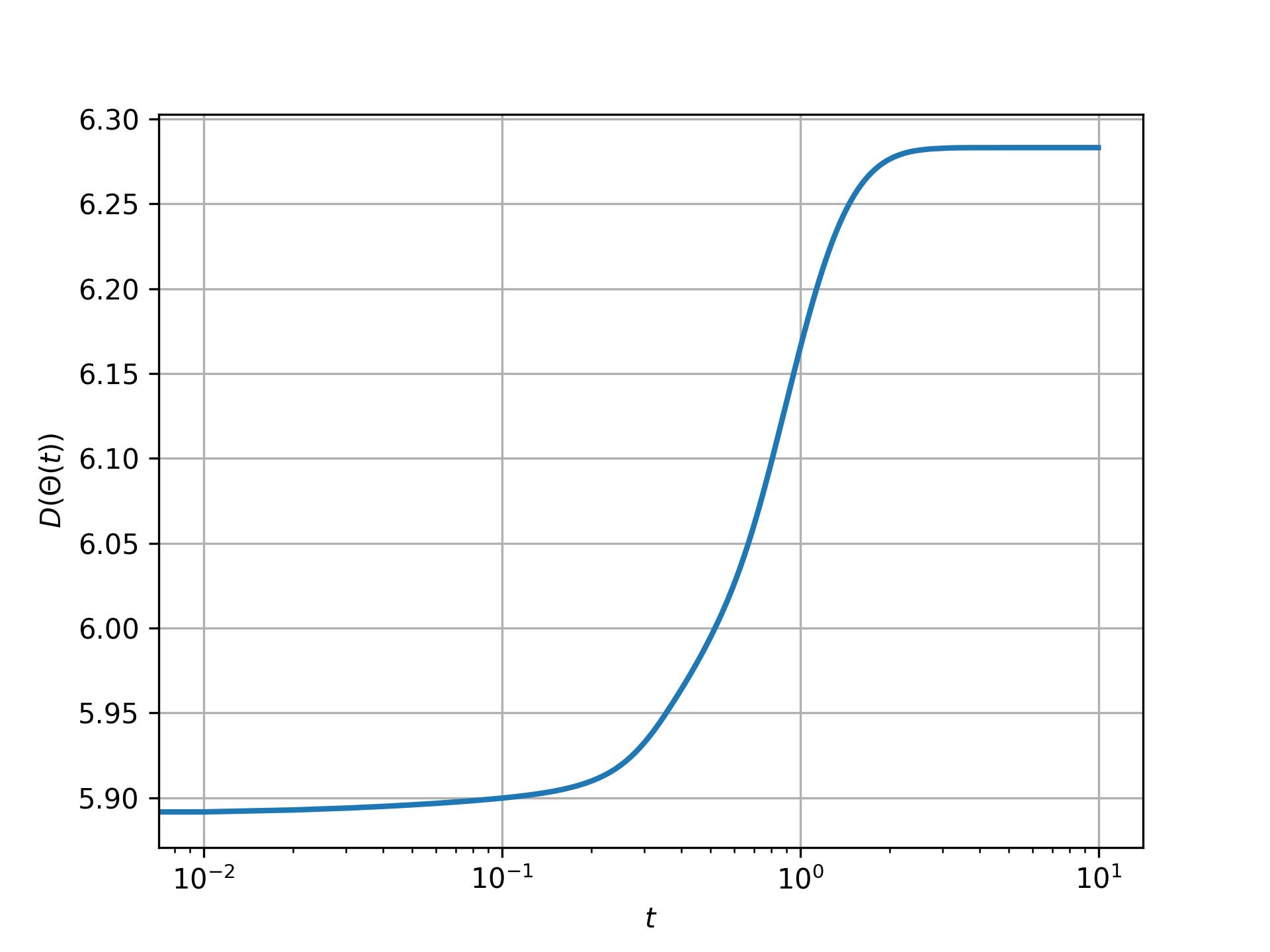}
	\caption{$D(\Theta(t))$}
\end{subfigure}%
\hfill
\begin{subfigure}{.5\textwidth}
	\centering
	\includegraphics[width=\textwidth]{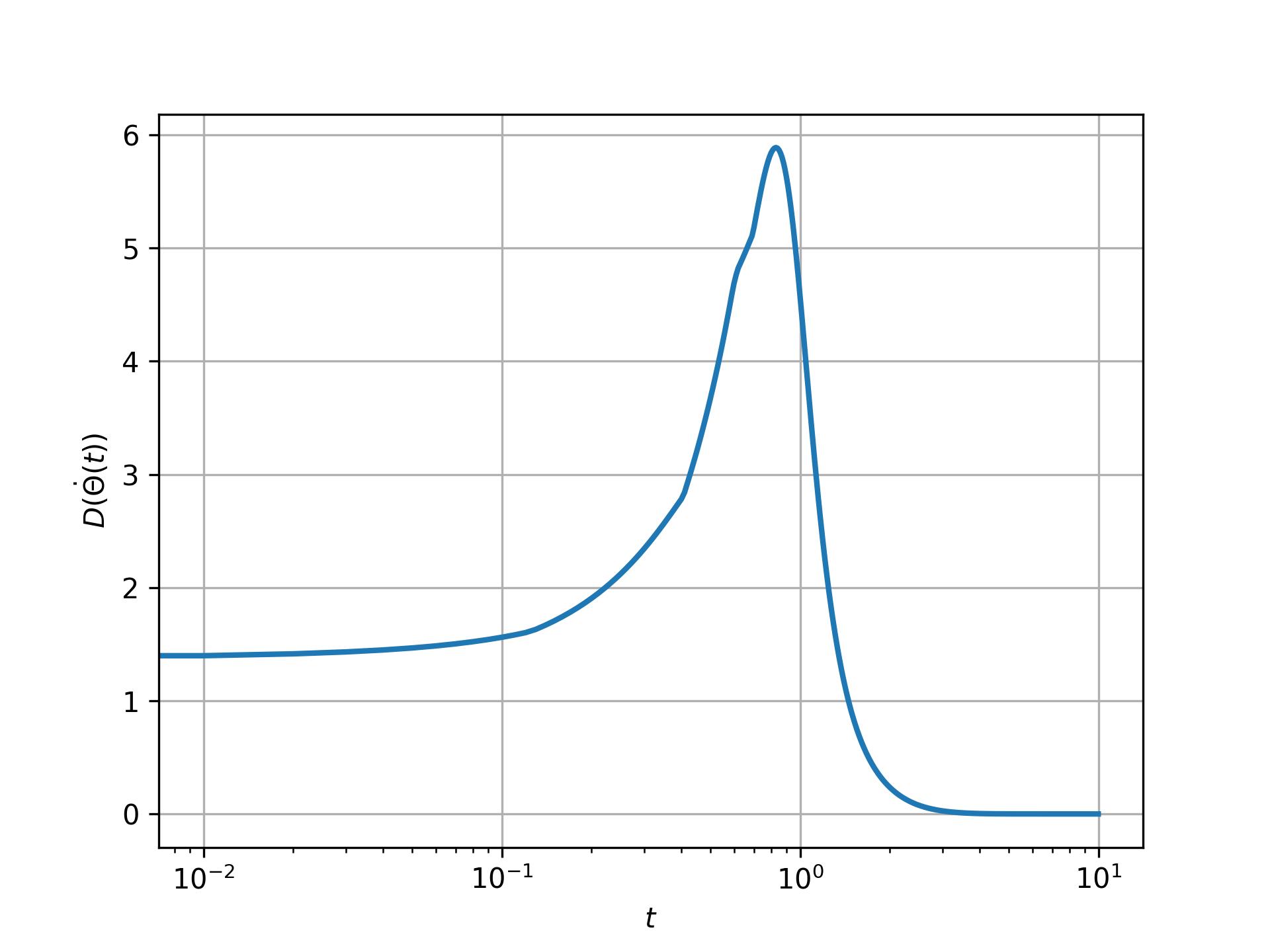}
	\caption{$D(\dot{\Theta}(t))$}
\end{subfigure}%

\caption{Solution of the SC Kuramoto model \eqref{eq_max_kuramoto} with $N=10$, $k=1$, $D(\Theta(0))=15\pi/8$ and $D(\Omega)=0$.}
\label{fig:identical_more_than_pi}
\end{figure}

Secondly, we show that the synchronized frequency can be larger than the largest natural frequency.
For example, suppose $N$ identical oscillators are uniformly distributed on the circle.
That is, $\omega_i = \omega$ and $\theta_i(0) = 2i\pi/N$ for all $1\leq i\leq N$.
Then, it can be checked that
\[
	\dot{\theta}_i(t) = \omega + k \sum_{1\leq j<N/2} \sin\frac{2j\pi}{N} > \omega, \quad \text{ for } i =1, 2, 3, \cdots, N.
\]
In Figure~\ref{fig:identical_not_max}, we consider $N=6$, $k=0.1$, and $\omega=0$.
The result confirms that the synchronized frequency is $2k\sin(\pi/3)\approx 0.1732$, larger than the maximal natural frequency $0$.

Thirdly, we consider non-identical oscillators with $k < D(\Omega)/\sin\delta$.
In Figure~\ref{fig:nonidentical_small_k}, we set $N=10$, $k=0.5$, $D(\Theta(0))=5\pi/6$, and $D(\Omega)=1$.
Theorem~\ref{thm_freq_sync} requires $k>D(\Omega)/\sin\delta = 2$, so the theorem does not 
apply.
We observe that the oscillators still exhibit complete frequency synchronization.

\begin{figure}[H]
\centering
\begin{subfigure}{.5\textwidth}
	\centering
	\includegraphics[width=\textwidth]{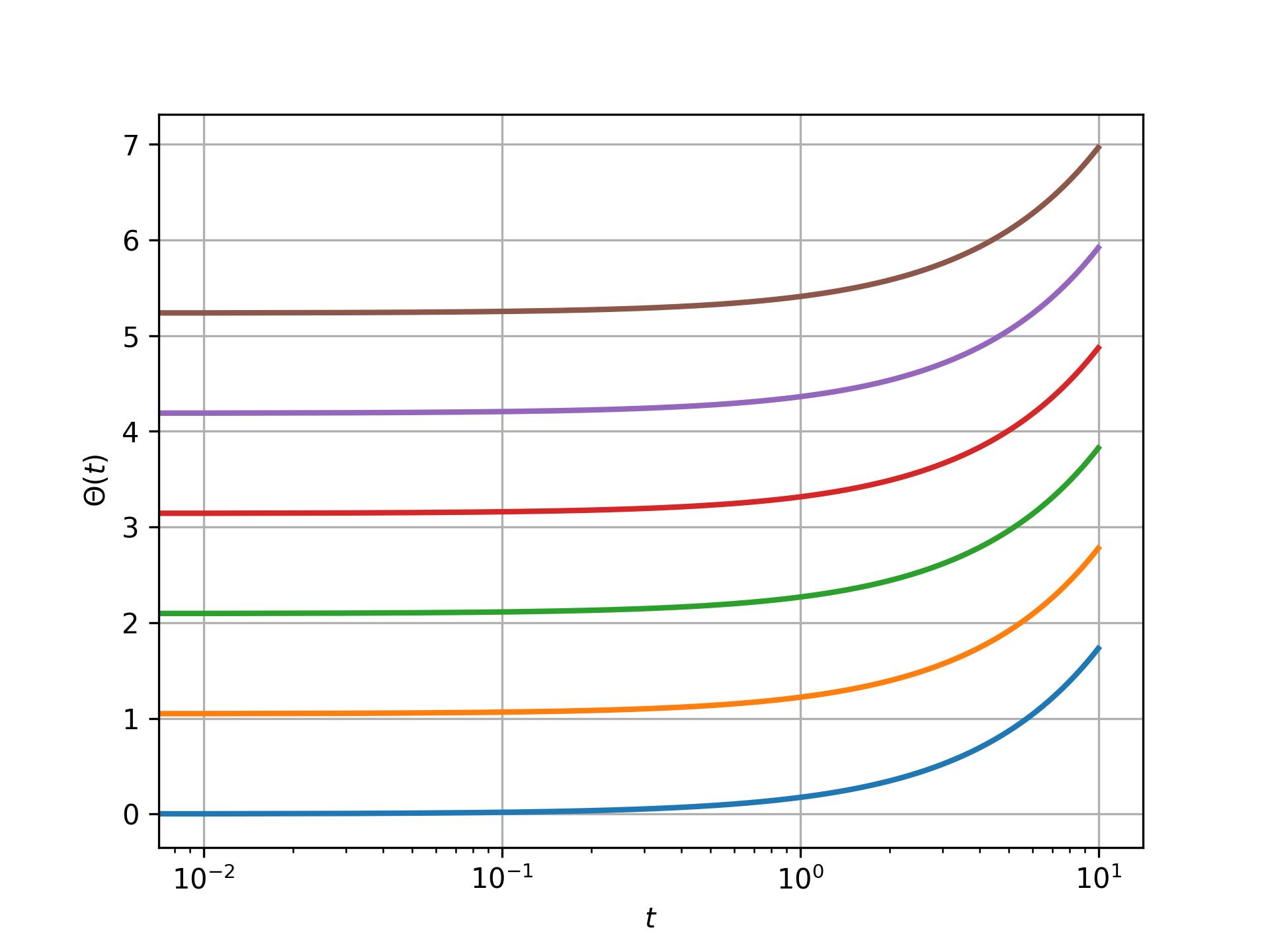}
	\caption{$\Theta(t)$.}
\end{subfigure}%
\hfill
\begin{subfigure}{.5\textwidth}
	\centering
	\includegraphics[width=\textwidth]{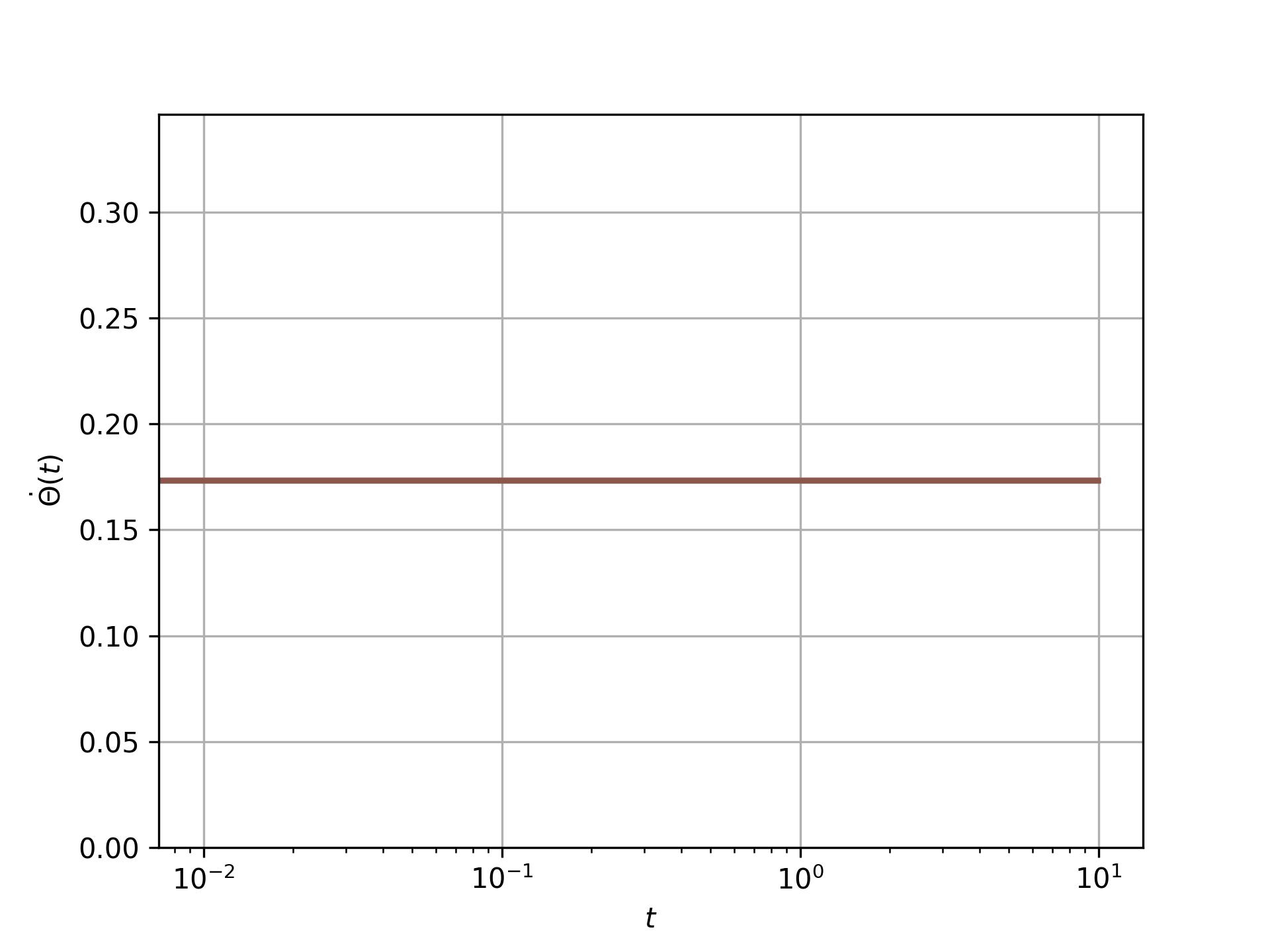}
	\caption{$\dot{\Theta}(t)$.}
\end{subfigure}%

\caption{Solution of the SC Kuramoto model \eqref{eq_max_kuramoto} with $N=6$, $k=0.1$, $\Theta(0) = (0, \pi/3, 2\pi/3, \ldots, 5\pi/3)$ and $D(\Omega)=0$.}
\label{fig:identical_not_max}
\end{figure}

\begin{figure}[t]
\centering
\begin{subfigure}{.5\textwidth}
	\centering
	\includegraphics[width=\textwidth]{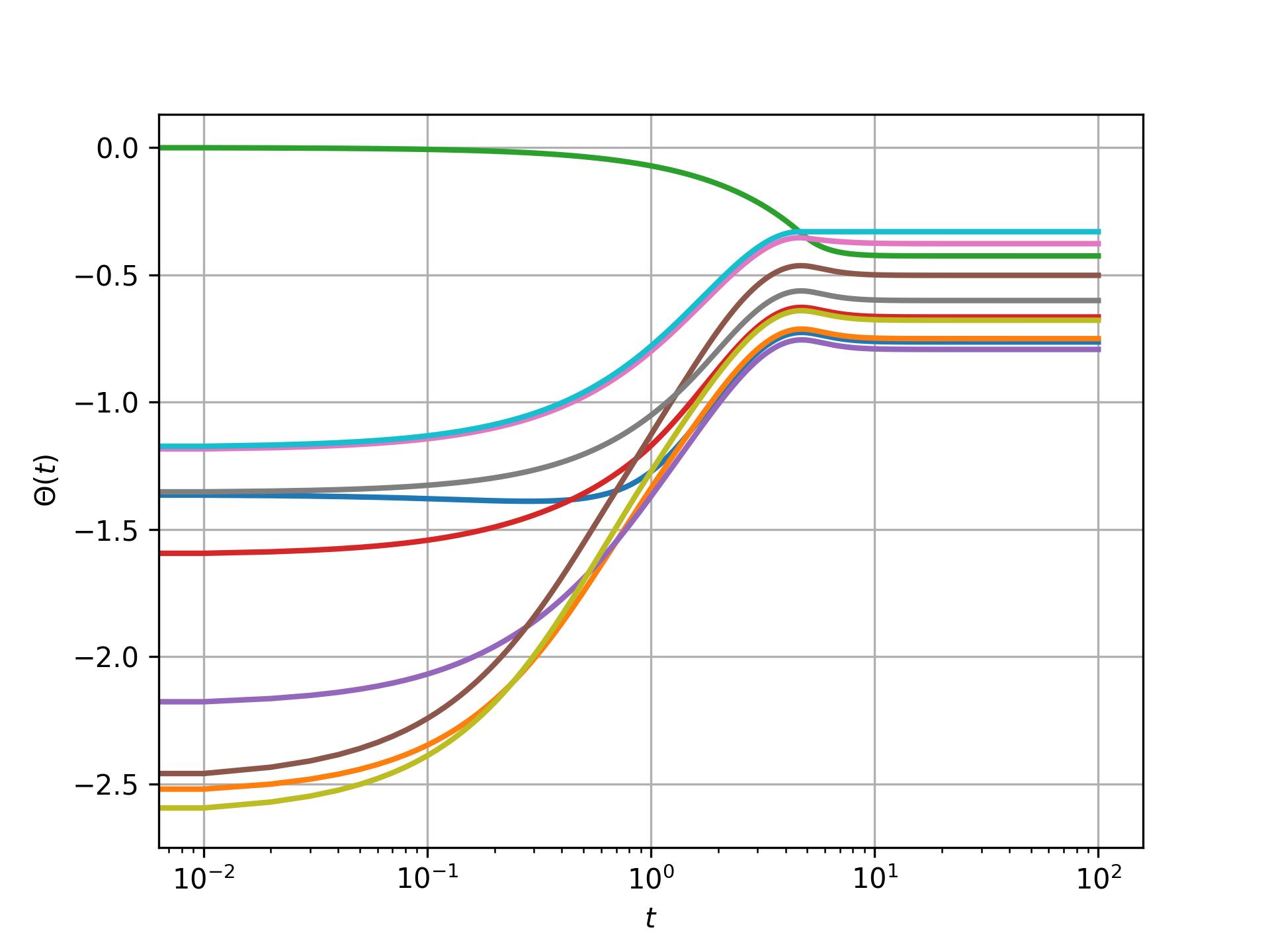}
	\caption{$\Theta(t)$.}
\end{subfigure}%
\hfill
\begin{subfigure}{.5\textwidth}
	\centering
	\includegraphics[width=\textwidth]{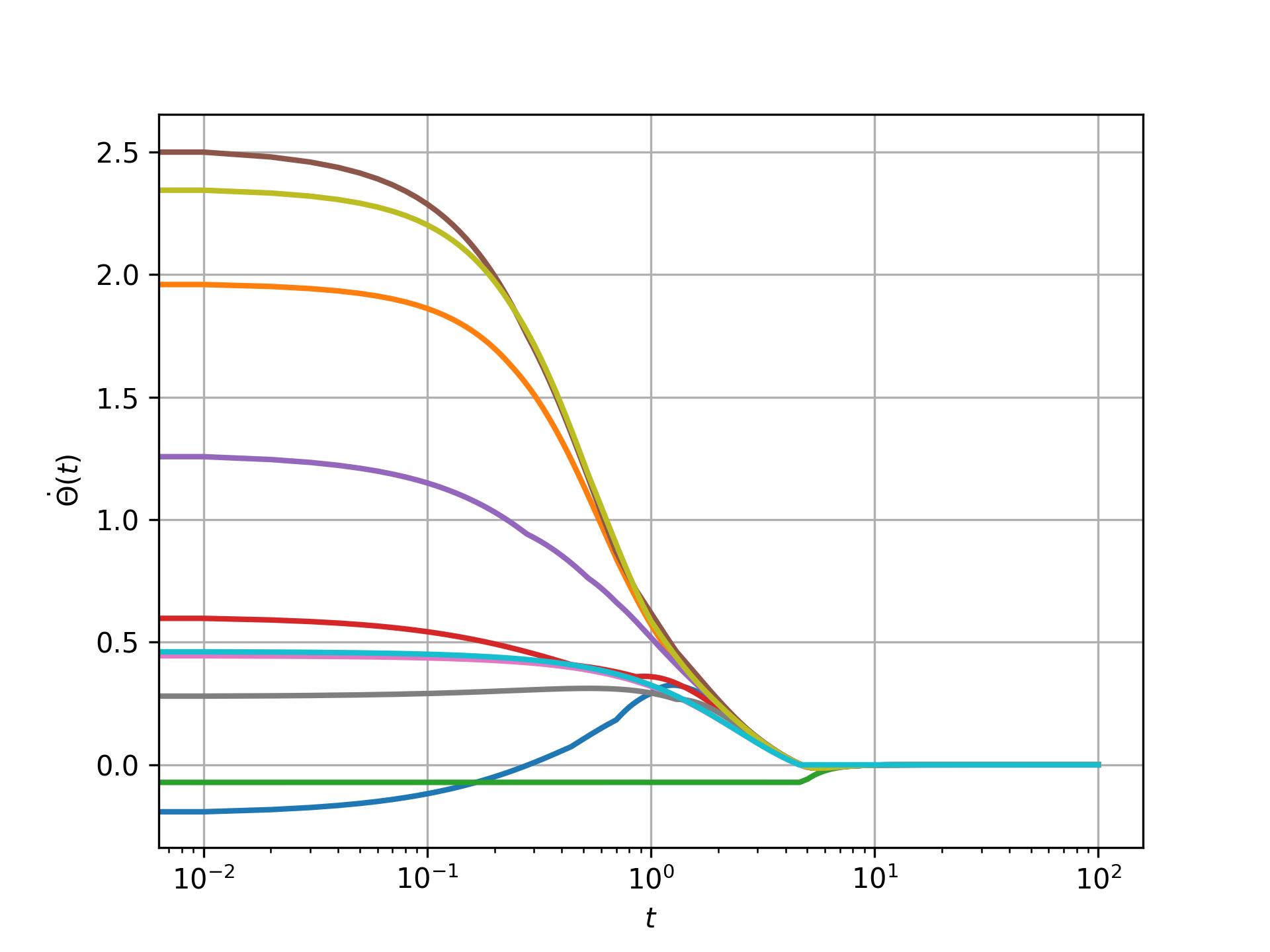}
	\caption{$\dot{\Theta}(t)$.}
\end{subfigure}%

\begin{subfigure}{.5\textwidth}
	\centering
	\includegraphics[width=\textwidth]{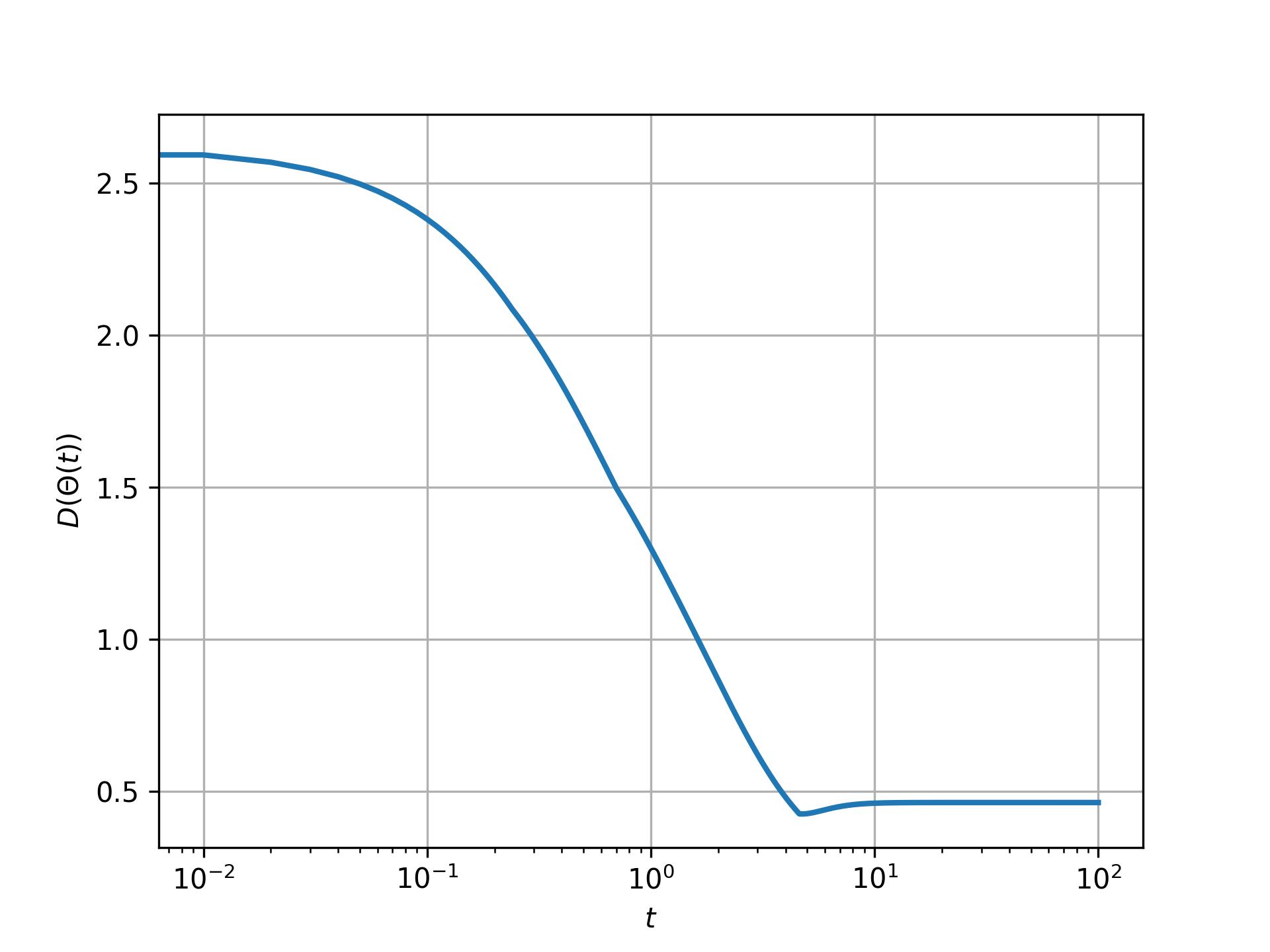}
	\caption{$D(\Theta(t))$}
\end{subfigure}%
\hfill
\begin{subfigure}{.5\textwidth}
	\centering
	\includegraphics[width=\textwidth]{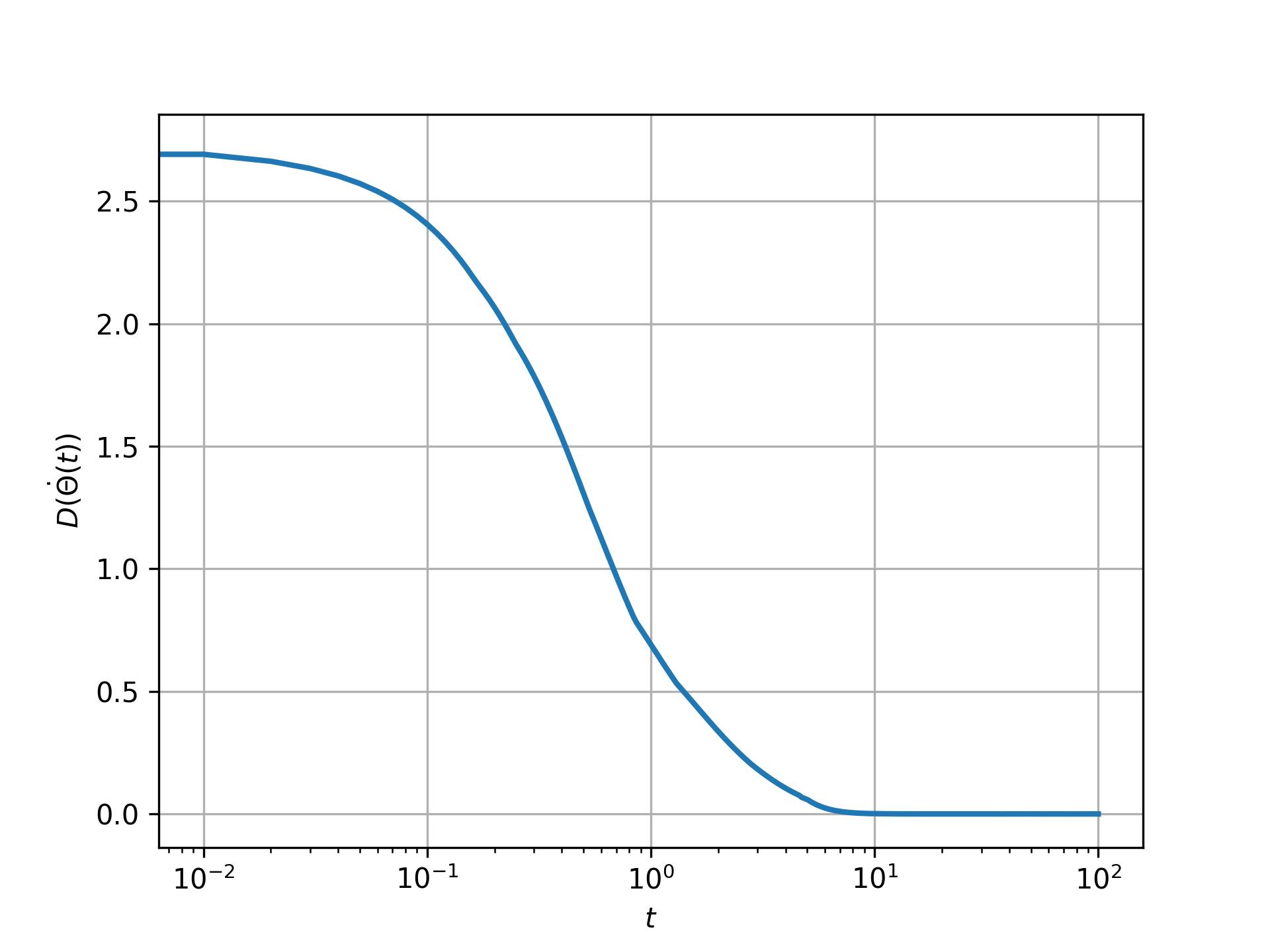}
	\caption{$D(\dot{\Theta}(t))$}
\end{subfigure}%

\caption{Solution of the SC Kuramoto model \eqref{eq_max_kuramoto} with $N=10$, $k=0.5$, $D(\Theta(0))=5\pi/6$ and $D(\Omega)=1$.}
\label{fig:nonidentical_small_k}
\end{figure}

\acks{C.~Hsia is supported in part by NCTS and NSTC with the grant numbers  112-2123-M-002-010-   and    109-2115-M-002-013-MY3.
C.~Tsai is supported by NSTC with the grant number 112-2628-E-002-019-MY3.}

\printbibliography

@article{Hansel1993a,
	author = {Hansel, D. and Mato, G. and Meunier, C.},
	journal = {Phys. Rev. E},
	pages = {3470--3477},
	title = {Clustering and slow switching in globally coupled phase oscillators},
	volume = {48},
	year = {1993}}

@article{Ho2024a,
	author = {Christine Ho and Laurent Jutras-Dub\'{e} and Michael Zhao and Gregor M\"{o}nke and Istv\'{a}n Z. Kiss and Paul Fran{\c c}ois and Alexander Aulehla},
	journal = {bioRxiv},
	title = {Nonreciprocal synchronization in embryonic oscillator ensembles},
	year = {2024}}

@article{Van-Hemmen1993a,
	author = {Van Hemmen, J. L. and Wreszinski, W. F.},
	journal = {J. Stat. Phys.},
	pages = {145--166},
	publisher = {Springer},
	title = {Lyapunov function for the {K}uramoto model of nonlinearly coupled oscillators},
	volume = {72},
	year = {1993}}

@article{Verwoerd2008a,
	author = {Verwoerd, Mark and Mason, Oliver},
	journal = {SIAM J. Appl. Dyn.},
	number = {1},
	pages = {134-160},
	title = {Global phase-locking in finite populations of phase-coupled oscillators},
	volume = {7},
	year = {2008}}

@article{Wang2017a,
	author = {Wang, Huobin and Han, Wenchen and Yang, Junzhong},
	journal = {Phys. Rev. E},
	pages = {022202},
	title = {Synchronous dynamics in the {K}uramoto model with biharmonic interaction and bimodal frequency distribution},
	volume = {96},
	year = {2017}}

@article{Yang2020a,
	author = {Yang, Seong-Gyu and Hong, Hyunsuk and Kim, Beom Jun},
	journal = {Sci. Rep.},
	number = {1},
	pages = {2516},
	title = {Asymmetric dynamic interaction shifts synchronized frequency of coupled oscillators},
	volume = {10},
	year = {2020}}

@article{Yang2020b,
	author = {Yang, Seong-Gyu and Park, Jong Il and Kim, Beom Jun},
	issue = {5},
	journal = {Phys. Rev. E},
	numpages = {5},
	pages = {052207},
	title = {Discontinuous phase transition in the {K}uramoto model with asymmetric dynamic interaction},
	volume = {102},
	year = {2020}}

@article{Skardal2011a,
	author = {Skardal, Per Sebastian and Ott, Edward and Restrepo, Juan G.},
	journal = {Phys. Rev. E},
	pages = {036208},
	title = {Cluster synchrony in systems of coupled phase oscillators with higher-order coupling},
	volume = {84},
	year = {2011}}

@article{Sakaguchi1986a,
	author = {Sakaguchi, Hidetsugu and Kuramoto, Yoshiki},
	journal = {Prog. Theor. Phys.},
	number = {3},
	pages = {576-581},
	title = {A soluble active rotater model showing phase transitions via mutual entertainment},
	volume = {76},
	year = {1986}}

@article{Strogatz2000a,
	author = {Steven H. Strogatz},
	journal = {Physica D},
	number = {1},
	pages = {1-20},
	title = {From {K}uramoto to {C}rawford: Exploring the onset of synchronization in populations of coupled oscillators},
	volume = {143},
	year = {2000}}

@article{Rodrigues2016a,
	author = {Francisco A. Rodrigues and Thomas K. DM. Peron and Peng Ji and J{\"u}rgen Kurths},
	journal = {Phys. Rep.},
	pages = {1-98},
	title = {The {K}uramoto model in complex networks},
	volume = {610},
	year = {2016}}

@book{Pikovsky2001a,
	author = {Arkady Pikovsky and Michael Rosenblum and J\"{u}rgen Kurths},
	publisher = {Cambridge University Press},
	title = {Synchronization: {A} {U}niversal {C}oncept in {N}onlinear {S}cience},
	year = {2001}}

@article{Li2014a,
	author = {Li, Keren and Ma, Shen and Li, Haihong and Yang, Junzhong},
	journal = {Phys. Rev. E},
	pages = {032917},
	title = {Transition to synchronization in a {K}uramoto model with the first- and second-order interaction terms},
	volume = {89},
	year = {2014}}

@article{Li2019a,
	author = {Li, Xue and Zhang, Jiameng and Zou, Yong and Guan, Shuguang},
	journal = {Chaos: An Interdisciplinary Journal of Nonlinear Science},
	number = {4},
	title = {{Clustering and Bellerophon state in Kuramoto model with second-order coupling}},
	volume = {29},
	year = {2019}}

@inproceedings{Kuramoto1975a,
	author = {Kuramoto, Yoshiki},
	booktitle = {Int. Symp. Math. Probl. Theor. Phys.},
	pages = {420--422},
	title = {Self-entrainment of a population of coupled non-linear oscillators},
	year = {1975}}

@book{Kuramoto1984a,
	author = {Yoshiki Kuramoto},
	publisher = {Springer},
	title = {Chemical {O}scillators, {W}aves, and {T}urbalance},
	year = {1984}}

@article{Komarov2013a,
	author = {Komarov, Maxim and Pikovsky, Arkady},
	journal = {Phys. Rev. Lett.},
	pages = {204101},
	title = {Multiplicity of singular synchronous states in the {K}uramoto model of coupled oscillators},
	volume = {111},
	year = {2013}}

@article{Hsia2019a,
	author = {Chun-Hsiung Hsia and Chang-Yeol Jung and Bongsuk Kwon},
	journal = {J. Differ. Equ.},
	number = {2},
	pages = {742-775},
	title = {On the synchronization theory of {K}uramoto oscillators under the effect of inertia},
	volume = {267},
	year = {2019}}

@article{Hsia2020a,
	author = {Chun-Hsiung Hsia and Chang-Yeol Jung and Bongsuk Kwon and Yoshihiro Ueda},
	journal = {J. Differ. Equ.},
	number = {12},
	pages = {7897-7939},
	title = {Synchronization of {K}uramoto oscillators with time-delayed interactions and phase lag effect},
	volume = {268},
	year = {2020}}

@article{Ha2014a,
	author = {Seung-Yeal Ha and Yongduck Kim and Zhuchun Li},
	journal = {Netw. Heterog. Media},
	number = {1},
	pages = {33-64},
	title = {Asymptotic synchronous behavior of {K}uramoto type models with frustrations},
	volume = {9},
	year = {2014}}

@article{Ha2013a,
	author = {Seung-Yeal Ha and Zhuchun Li and Xiaoping Xue},
	journal = {J. Differ. Equ.},
	number = {10},
	pages = {3053-3070},
	title = {Formation of phase-locked states in a population of locally interacting Kuramoto oscillators},
	volume = {255},
	year = {2013}}

@article{Gong2019a,
	author = {Gong, Chen Chris and Pikovsky, Arkady},
	journal = {Phys. Rev. E},
	pages = {062210},
	title = {Low-dimensional dynamics for higher-order harmonic, globally coupled phase-oscillator ensembles},
	volume = {100},
	year = {2019}}

@article{Eydam2017a,
	author = {Eydam, Sebastian and Wolfrum, Matthias},
	journal = {Phys. Rev. E},
	pages = {052205},
	title = {Mode locking in systems of globally coupled phase oscillators},
	volume = {96},
	year = {2017}}

@article{Dong2013a,
	author = {Jiu-Gang Dong and Xiaoping Xue},
	journal = {Commun. Math. Sci.},
	number = {2},
	pages = {465--480},
	title = {Synchronization analysis of {K}uramoto oscillators},
	volume = {11},
	year = {2013}}

@article{Dorfler2011a,
	author = {D\"{o}rfler, Florian and Bullo, Francesco},
	journal = {SIAM J. Appl. Dyn},
	number = {3},
	pages = {1070-1099},
	title = {On the critical coupling for {K}uramoto oscillators},
	volume = {10},
	year = {2011}}

@article{Delabays2019a,
	author = {Delabays, Robin},
	journal = {Chaos: An Interdisciplinary Journal of Nonlinear Science},
	number = {11},
	title = {{Dynamical equivalence between Kuramoto models with first- and higher-order coupling}},
	volume = {29},
	year = {2019}}

@article{Daido1996a,
	author = {Hiroaki Daido},
	journal = {Physica D: Nonlinear Phenomena},
	number = {1},
	pages = {24-66},
	title = {Onset of cooperative entrainment in limit-cycle oscillators with uniform all-to-all interactions: bifurcation of the order function},
	volume = {91},
	year = {1996}}

@article{Daido2016a,
	author = {Daido, Hiroaki; Nishio, Kazuho},
	journal = {Phys. Rev. E},
	pages = {052226},
	title = {Bifurcation and scaling at the aging transition boundary in globally coupled excitable and oscillatory units},
	volume = {93},
	year = {2016}}

@article{Chopra2009a,
	author = {Chopra, Nikhil and Spong, Mark W.},
	journal = {IEEE Trans. Automat. Contr.},
	number = {2},
	pages = {353-357},
	title = {On exponential synchronization of {K}uramoto oscillators},
	volume = {54},
	year = {2009}}

@article{Choi2012a,
	author = {Young-Pil Choi and Seung-Yeal Ha and Sungeun Jung and Yongduck Kim},
	journal = {Physica D},
	number = {7},
	pages = {735-754},
	title = {Asymptotic formation and orbital stability of phase-locked states for the {K}uramoto model},
	volume = {241},
	year = {2012}}

@article{Benedetto2015a,
	author = {Dario Benedetto and Emanuele Caglioti and Umberto Montemagno},
	journal = {Commun. Math. Sci.},
	number = {7},
	pages = {1775--1786},
	title = {On the complete phase synchronization for the {K}uramoto model in the mean-field limit},
	volume = {13},
	year = {2015}}

@article{Acebron2005a,
	author = {Acebr\'on, Juan A. and Bonilla, L. L. and P\'erez Vicente, Conrad J. and Ritort, F\'elix and Spigler, Renato},
	issue = {1},
	journal = {Rev. Mod. Phys.},
	numpages = {0},
	pages = {137--185},
	title = {The {K}uramoto model: A simple paradigm for synchronization phenomena},
	volume = {77},
	year = {2005}}

\end{document}